\documentclass[12pt,reqno]{amsart}
\usepackage{amssymb}
\usepackage{amsfonts}
\usepackage{amsmath}
\usepackage{mathrsfs}
\usepackage{tikz}
\usetikzlibrary{shapes.geometric,arrows,decorations.markings,calc}
\newtheorem{theorem}{Theorem}[section]
\newtheorem{prop}[theorem]{Proposition}
\newtheorem{lemma}[theorem]{Lemma}
\newtheorem{cor}[theorem]{Corollary}

\theoremstyle{definition}

\newtheorem{example}[theorem]{Example}
\newtheorem{defn}[theorem]{Definition}

\newtheorem{rmk}[theorem]{Remark}
\newtheorem{rmks}[theorem]{Remarks}
\newtheorem{notn}[theorem]{Notation}
\newcommand{\qnum}[2]{[#1]_{#2}}

\DeclareMathOperator{\ch}{char}
\DeclareMathOperator{\Spec}{Spec}
\DeclareMathOperator{\aut}{Aut}
\DeclareMathOperator{\Endo}{End}
\DeclareMathOperator{\Jac}{Jac}
\DeclareMathOperator{\Pz}{PZ}
\newcommand{\Z}{{\mathbb Z}}
\newcommand{\C}{{\mathbb C}}
\newcommand{\N}{{\mathbb N}}
\newcommand{\Q}{{\mathbb Q}}
\newcommand{\K}{{\mathbb K}}

\begin{document}

\title{Connected quantized Weyl algebras and quantum cluster algebras}

\author{Christopher D. Fish}
\address{School of Mathematics and Statistics\\
University of Sheffield\\
Hicks Building\\
Sheffield S3~7RH\\
UK}
\email{christopher.fish@cantab.net}
\author{David A. Jordan}
\address{School of Mathematics and Statistics\\
University of Sheffield\\
Hicks Building\\
Sheffield S3~7RH\\
UK}
\email{d.a.jordan@sheffield.ac.uk}

\date{\today}



\thanks{Some of the results in this paper appear in the University of Sheffield PhD thesis of the first author, supported by the Engineering and Physical Sciences Research Council of the UK. We thank Vladimir Dotsenko for his helpful comments at an early stage of this project, the referee for many helpful suggestions of improvements in the exposition and Bach V Nguyen for pointing out some errors in our original manuscript.}

\begin{abstract}
For an algebraically closed field $\K$, we investigate a class of noncommutative $\K$-algebras called \emph{connected quantized Weyl algebras}. Such an algebra has a  PBW basis for a set of generators $\{x_1,\dots,x_n\}$ such that each pair satisfies a relation of the form $x_ix_j=q_{ij}x_jx_i+r_{ij}$, where $q_{ij}\in \K^*$ and $r_{ij}\in \K$, with, in some sense, sufficiently many pairs for which $r_{ij}\neq 0$.  For such an algebra it turns out that there is a single parameter $q$ such that each $q_{ij}=q^{\pm1}$.
Assuming that $q\neq \pm1$, we classify connected quantized Weyl algebras, showing that there are two types \emph{linear} and \emph{cyclic}. When $q$ is not a root of unity we determine the prime spectra for each type. The linear case is the easier, although the result depends on the parity of $n$, and all prime ideals are completely prime. In the cyclic case, which can only occur if $n$ is odd, there are prime ideals for which the factors have arbitrarily large Goldie rank.

We apply connected quantized Weyl algebras to obtain presentations of two classes of quantum cluster algebras. Let $n\geq 3$ be an odd integer. We present the quantum cluster algebra of a Dynkin quiver of type $A_{n-1}$ as a factor of a linear connected quantized Weyl algebra by an ideal generated by a central element.  We also consider the quiver $P_{n+1}^{(1)}$ identified by Fordy and Marsh in their analysis of periodic quiver mutation. When $n$ is odd, we show that the quantum cluster algebra of this quiver is generated by a cyclic connected quantized Weyl algebra in $n$ variables and one further generator. We also present it as the factor of an iterated skew polynomial algebra in $n+2$ variables by an ideal generated by a central element. For both classes, the quantum cluster algebras are simple noetherian.

We establish Poisson analogues of the results on prime ideals and quantum cluster algebras. We determine the Poisson prime spectra for the semiclassical limits of the linear and cyclic connected quantized Weyl algebras and show that, when $n$ is odd, the cluster algebras of $A_{n-1}$ and $P_{n+1}^{(1)}$ are simple Poisson algebras that can each be presented as a Poisson factor of a polynomial algebra, with an appropriate Poisson bracket, by a principal ideal generated by a Poisson central element.
\end{abstract}

\maketitle

\section{Introduction}
This paper is mostly in the context of noncommutative ring theory, in particular skew polynomial rings,  classification of  prime ideals and applications to quantum cluster algebras. The original motivation  can be traced back to the classification of mutation periodic quivers by Fordy and Marsh \cite{fordymarsh} and to a Poisson algebra $P$ introduced by Fordy \cite{fordy} in a further study of some such quivers.

The Poisson algebra $P$ is a polynomial algebra in an odd number of indeterminates $x_1,\dots, x_n$ and it may be helpful to think of these arranged cyclically, so that $x_1$ is adjacent to $x_n$ as well as to $x_2$.  Up to a factor of $2$, the Poisson bracket is such that
\begin{equation}
\{x_i,x_{i+1}\}=x_ix_{i+1}-1,\quad 1\leq i\leq n, \label{padj}
\end{equation} where $x_{n+1}$ should be interpreted as $x_1$, and, for $1\leq i,j\leq n$ with $j>i+1$,
\begin{equation}\{x_i,x_j\}=\begin{cases}
x_ix_j&\text{ if }j-i\text{ is odd,}\\
-x_ix_j&\text{ if }j-i\text{ is even.}
\end{cases}\label{pnonadj}
\end{equation}

In the sense of \cite[Chapter III.5]{BGl}, this algebra is quantized by the algebra $C_n^q$ generated by $x_1,\dots, x_n$ subject to the relations
\begin{equation}
x_ix_{i+1}-qx_{i+1}x_i=1-q, \quad 1\leq i\leq n, \label{nonpadj}
\end{equation} where $x_{n+1}$ should again be interpreted as $x_1$, and, for $1\leq i,j\leq n$ with $j>i+1$,
\begin{equation}x_ix_j=\begin{cases}
qx_jx_i&\text{ if }j-i\text{ is odd,}\\
q^{-1}x_jx_i&\text{ if }j-i\text{ is even.}
\end{cases}\label{nonpnonadj}
\end{equation}

We shall interpret the relation \eqref{nonpadj} as the defining relation for the quantized Weyl algebra $A_1^q$ generated by $x_i$ and $x_{i+1}$. This is more commonly written
\begin{equation}
x_ix_{i+1}-qx_{i+1}x_i=1,\quad 1\leq i\leq n, \label{usualqw}
\end{equation}
but, unless $q=1$, the two are isomorphic (by a change of variable) and \eqref{usualqw} is less satisfactory both from the point of view of quantization, as it gives a noncommutative algebra on setting $q=1$, and from the point of view of symmetry, as it is equivalent  not to
$x_{i+1}x_i-q^{-1}x_ix_{i+1}=1$ but to $x_{i+1}x_i-q^{-1}x_ix_{i+1}=-q^{-1}$.

It is possible to construct $C_n^q$ as an iterated skew polynomial algebra in $x_1, x_2,\dots, x_n$ over the base field $\K$. If $1\leq m<n$ the intermediate iterated skew polynomial algebra in $x_1, x_2,\dots, x_m$ will be denoted $L_m^q$. When $n$ is odd, $L_n^q$ and $C_n^q$ both exist and are skew polynomial rings over $L_{n-1}^q$.   As iterated skew polynomial algebras over the base field, $L_n^q$ and, if $n$ is odd, $C_n^q$ are noetherian domains with PBW-bases.

Motivated by the algebras $L_n^q$ and $C_n^q$, we define a {\emph connected quantized Weyl algebra} to be a $\K$-algebra with a finite set $\{x_1,\dots,x_n\}$ of generators such that \begin{itemize}
\item
  each pair of generators satisfies a relation of the form
\[x_ix_j=q_{ij}x_jx_i+r_{ij},\]
where $q_{ij}\in \K^*$ and $r_{ij}\in \K$,
\item  the standard monomials $x_1^{a_1}x_2^{a_2}\dots x_n^{a_n}$ form a PBW basis,
\item  the graph  $G$ with vertices $x_1,x_2,\dots,x_n$, in which there is an edge between $x_i$ and $x_j$ if and only if $r_{ij}\neq 0$,  is connected.  \end{itemize}

    In Section 2 we shall see that, provided at least one $q_{ij}\neq \pm 1$,
$L_n^q$ and $C_n^q$ are the only connected quantized Weyl algebras.

In Sections 3 and 4, using a deleting derivations algorithm similar to that applied to quantum matrices by Cauchon \cite{cauchon}, we determine, when $q$ is not a root of unity, the prime spectra of $L_n^q$ and, when $n$ is odd,  $C_n^q$. In the case of $L_n^q$,  the hypotheses of
\cite[Theorem 2.3]{goodletzcompprime} are satisfied so all prime ideals are completely prime but we shall see that $C_n^q$ does have prime ideals that are not completely prime. In $L_n^q$ there is a sequence of elements $z_1,\dots, z_n$, defined by the formula $z_j=z_{j-1}x_j-z_{j-2}$, where $z_0=1$ and $z_{-1}=0$, such that $z_n$ is central if $n$ is odd and normal, but not central, if $n$ is even. In the odd case, the non-zero prime ideals of $L_n^q$ are the ideals of the form $(z_n-\lambda)L_n^q$, $\lambda\in \K$, and in the even case they are $z_nL_n^q$ and the ideals of the form $z_nL_n^q+(z_{n-1}-\lambda)L_n^q$, $0\neq\lambda\in \K$. There is a similar sequence in $C_n^q$ but with $z_n$ replaced by a central element $\Omega$. The prime ideals of $C_n^q$ are the ideals of the form $(\Omega-\lambda)C_n^q$, $\lambda\in \K$, and, for each positive integer $m$, two prime ideals $F_{m,1}$ and $F_{m,-1}$, such that the factors $C_n^q/F_{m,\pm 1}$ have Goldie rank $m$.    Thus the prime spectrum of $L_n^q$ is akin to those of $U(sl_2)$ and $U_q(sl_2)$ but, unless $n=3$, the exceptional maximal ideals are not annihilators of finite-dimensional simple modules.

Section 5 determines the  $\K$-automorphism groups of $L_n^q$ and $C_n^q$ when $q\neq\pm 1$. Whereas $\aut_\K(L_n^q)$ is isomorphic to the multiplicative group $\K^*$, with each $\lambda\in \K^*$ corresponding to an automorphism with $x_i\mapsto
\lambda^{(-1)^i}x_i$, $\aut_\K(C_n^q)$ is cyclic of order $2n$ generated by the product of the $\K$-automorphism of order $n$ such that each $x_i\mapsto x_{i+1}$, where $x_{n+1}=x_1$, and the automorphism of order $2$ such that each $x_i\mapsto -x_i$.

In Section 6 we apply connected quantized Weyl algebras to quantum cluster algebras. Useful references for such algebras include \cite{BerFomZel,goodyak,GrabLaun,Keller}.
Although there are several papers, for example \cite{goodyak,GrabLaun}, showing that given noncommutative algebras have quantum cluster algebra structures, there are not many in which a quantum cluster algebra is determined given a particular quiver. For two classes of quivers, we relate the quantum cluster algebra to connected quantized Weyl algebras and obtain presentations in terms of generators and relations. For an even positive integer $m$ we present the quantum cluster algebra for a Dynkin quiver of type $A_m$ as a factor of the linear connected quantized Weyl algebra $L_{m+1}^{q}$.
For an odd integer $n\geq 3$ we present the quantum cluster algebra for the periodic quiver denoted $P_{n+1}^{(1)}$ in \cite{fordymarsh} as
an extension of the cyclic connected quantized Weyl algebra $C_{n}^{q^2}$, requiring one further generator, and also as a factor, by the principal ideal generated by a central element, of an iterated skew polynomial ring in $n+2$ variables over the base field $\K$.

Sections 7 and 8 present Poisson analogues of earlier results. In Section 7 the Poisson prime spectra of the semiclassical limits of $L_n^q$ and $C_n^q$ are determined and in Section 8 the cluster algebras of $A_{n-1}$ and $P_{n+1}^{(1)}$ are presented as factors, by  Poisson ideals generated by a Poisson central element, of polynomial algebras with appropriate Poisson brackets.

\section{Connected quantized Weyl algebras}\label{prelims}
Throughout  $\K$ will denote an algebraically closed field and $q\in \K^*$. We make a fixed choice of one of the square roots of $q$ in $\K$ and denote it by $q^\frac{1}{2}$.

\begin{notn}\label{Aoneq}The (co-ordinate ring of the)
 quantum plane, $R_q$, is the $\K$-algebra generated by $x$ and $y$ subject to the relation
$xy=qyx$. There is symmetry, up to the transposition of $q$ and $q^{-1}$, as the relation can be rewritten $yx=q^{-1}xy$. It is well-known that $R_q$ is the skew polynomial ring $\K[y][x;\alpha_q]$ where $\alpha_q$ is the $\K$-automorphism of $\K[y]$ such that $\alpha_q(y)=qy$. As such, it has a PBW basis $\{y^ix^j:i,j\geq 0\}$. By symmetry, $R_q=\K[x][y;\beta_{q^{-1}}]$ where $\beta_{q^{-1}}$ is the $\K$-automorphism of $\K[x]$ such that $\beta_{q^{-1}}(x)=q^{-1}x$ and has a PBW basis $\{x^iy^j:i,j\geq 0\}$.

By the first quantized Weyl algebra $A_1^{q}(\K)$, we mean the $\K$-algebra generated by $x$ and $y$ subject to the relation $
xy-qyx=1-q.$ With $\alpha_q$ and $\beta_{q^{-1}}$ as in the above discussion of the quantum plane, $A_1^{q}(\K)$ has presentations as the skew polynomial ring $\K[y][x;\alpha_q, \delta]$, where $\delta(y)=1-q$, and as the skew polynomial ring $\K[x][y;\beta_{q^{-1}}, \delta^\prime]$, where  $\delta^\prime(y)=1-q^{-1}$.  As such, it has PBW bases $\{y^ix^j:i,j\geq 0\}$ and $\{x^iy^j:i,j\geq 0\}$.
\end{notn}

\begin{rmk}\label{semiclass}
The reader may be more familiar with the relation $xy-qyx=1$ here. Unless $q=1$, the two give isomorphic algebras. Indeed, if $r\in \K$,
and $xy-qyx=1$ and $x^\prime=rx$ then $x^\prime y-qyx^\prime=r$ so, by changing generators, the right hand side of the relation $xy-qyx=1$ can be replaced by any non-zero scalar $r$. There are two advantages in choosing $1-q$ in this role. The first is that the relation $xy-qyx=1-q$ can be rewritten $yx-q^{-1}xy=1-q^{-1}$, giving symmetry, up to the transposition of $q$ and $q^{-1}$.  The second is that setting $q=1$ yields the commutative algebra $\K[x,y]$, giving rise to a Poisson bracket on $\K[x,y]$ with $\{x,y\}=xy-1$. This Poisson bracket arises from the quantization procedure outlined in \cite[III.5.4]{BGl} with $A=\K[y,Q^{\pm 1}][x;\alpha,\rho]$ and $h=Q-1$, where $\alpha$ is the $\K$-automorphism of $\K[y,Q^{\pm 1}]$ such that $\alpha(y)=Qy$ and $\alpha(Q)=Q$, and $\rho$ is the $\alpha$-derivation of $\K[y]$ such that $\rho(y)=1-Q$ and $\delta(Q)=0$, whence $Q$ and $h$ are central, $xy-Qyx=(1-Q)$, $xy-yx=h(xy-1)$, $A/hA\simeq \K[x,y]$ and, if $q\neq 1$, $A/(Q-q)A\simeq A_1^{q}(\K)$.
\end{rmk}

 Loosely speaking, a connected quantized Weyl algebra is a finitely generated $\K$-algebra with a PBW basis in which any two generators satisfy a quantum plane relation or a quantized Weyl relation and there are sufficiently many of the latter. The formal definition is as follows.

\begin{defn}\label{cqwa}
By a \emph{connected quantized Weyl algebra} over $\K$, we shall mean a $\K$-algebra with generators $x_1,x_2,\dots,x_n$, $n\geq 2$, satisfying the following properties:
\begin{enumerate}
\item there are scalars $q_{ij}\in \K^*$ and $r_{ij}\in \K$, $1\leq i\neq j\leq n$, with each $q_{ji}=q_{ij}^{-1}$ and each $r_{ji}=-q_{ij}^{-1}r_{ij}$, such that the relation
\[x_ix_j=q_{ij}x_jx_i+r_{ij}\] holds;
\item the standard monomials $x_1^{a_1}x_2^{a_2}\dots x_n^{a_n}$ form a PBW basis for $R$;
\item the graph with vertices $x_1,x_2,\dots,x_n$, in which there is an edge between $x_i$ and $x_j$ if and only if $r_{ij}\neq 0$, is connected.
      \end{enumerate}
\end{defn}

\begin{rmk} The conditions on $q_{ji}$ and $r_{ji}$ ensure that the relations $x_ix_j=q_{ij}x_jx_i+r_{ij}$ and $x_jx_i=q_{ji}x_ix_j+r_{ji}$
are equivalent. The relations $x_ix_j=q_{ij}x_jx_i+r_{ij}$ ensure that the standard monomials $x_1^{a_1}x_2^{a_2}\dots x_n^{a_n}$ span $R$. The PBW condition therefore ensures that the relations $x_ix_j=q_{ij}x_jx_i+r_{ij}$ form a complete set of defining relations for $R$.

When $n=2$, up to isomorphism, the quantized Weyl algebra  $A_i^q$ is the only connected quantized Weyl algebra over $\K$. We can take $q_{12}=q$ and $r_{12}=1-q$ or, if $q=1$, $r_{12}=1$.
\end{rmk}

\begin{rmk}
We can orientate and label the graph in the definition to carry more information. As shown in Figure~\ref{edge}, we  orientate an edge representing a relation
$x_ix_j=q_{ij}x_jx_i+r_{ij}$ with $r_{ij}\neq 0$ from $x_i$ to $x_j$ with label $q_{ij}$ (or from $x_j$ to $x_i$ with label $q_{ji}$, but not both).

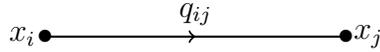
\begin{figure}[h]
\begin{center}
\begin{tikzpicture}[]
\begin{scope}[thick,decoration={
    markings,
    mark=at position 0.5 with {\arrow{>}}}
    ]
\draw[postaction={decorate}](180:2)--(0:2);
\end{scope}
\draw (0:2.3) node{$x_j$};
\draw (180:2.3) node{$x_i$};
\draw (0:2) node{$\bullet$};
\draw (180:2) node{$\bullet$};
\draw (90:0.3) node{$q_{ij}$};
\end{tikzpicture}
\end{center}
\caption{orientation of edges}
\label{edge}
\end{figure}
\end{rmk}

\begin{example}\label{Lnq}
Let $n\geqslant 1$ and let $q\in \K^*\backslash\{1\}$. Let $L_n^q$ denote the $\K$-algebra generated by $x_1, x_2,\dots, x_n$ subject to the relations
\begin{alignat}{3}\label{Lpres1}
x_ix_{i+1}-qx_{i+1}x_i&=1-q,&\quad&1\leqslant i\leqslant n-1,\\
\label{Lpres2}
x_ix_j-qx_jx_i&=0,&\quad& i\geqslant 1,\,i+1<j\leqslant n,\, j-i\text{ odd},\\
\label{Lpres3}
x_ix_j-q^{-1}x_jx_i&=0,&\quad& i\geqslant 1,\,i+1<j\leqslant n,\, j-i\text{ even}.
\end{alignat}

In particular $L_1^q=\K[x_1]$ and $L_2^q=A_1^q$. Using \cite[Proposition 1]{fddaj}, one can show, inductively, that, for $n\geq 2$, $L_n^q$ is the skew polynomial ring $L_{n-1}^q[x_n;\tau_n,\delta_n]$ for a $\K$-automorphism $\tau_n$ and a $\tau_n$-derivation $\delta_n$ of $L_{n-1}^q$  such that
 \begin{alignat*}{2}
\tau_n(x_j)&=q^{(-1)^{n-j}}x_j,  &\quad 1\leqslant j\leqslant n-1,\\
\delta_n(x_j)&=0,  &\quad 1\leqslant j\leqslant n-2,\\
\delta_n(x_{n-1})&=1-q^{-1}.&
\end{alignat*}
Informally, it suffices to show that $\tau_n$ and $\delta_n$ respect the defining relations of $L_{n-1}^q$. More formally, one can write $L_{n-1}^q$ as a factor $F/I$ of the free algebra over $\K$ on $n-1$ generators and show that $\tau_n$ and $\delta_n$ are induced by an
appropriate automorphism $\Gamma_n$ and $\Gamma_n$-derivation  $\Delta_n$ of $F$ such that $\Gamma_n(I)=I$ and $\Delta_n(I)\subseteq I$.
\begin{figure}[h]
\begin{center}
\begin{tikzpicture}[]
\begin{scope}[thick,decoration={
    markings,
    mark=at position 0.5 with {\arrow{>}}}
    ]
\draw[postaction={decorate}](-2,0)--(-1,0);
\draw[postaction={decorate}](-1,0)--(-0,0);
\draw[postaction={decorate}](0,0)--(1,0);
\draw[postaction={decorate}](1,0)--(2,0);

\end{scope}
\draw (-2,-0.3) node{$x_1$};
\draw (-1,-0.3) node{$x_2$};
\draw (0,-0.3) node{$x_3$};
\draw (1,-0.3) node{$x_4$};
\draw (2,-0.3) node{$x_5$};

\draw (-1.5,0.3) node{$q$};
\draw (-0.5,0.3) node{$q$};
\draw (0.5,0.3) node{$q$};
\draw (1.5,0.3) node{$q$};

\draw (-2,0) node{$\bullet$};
\draw (-1,0) node{$\bullet$};
\draw (0,0) node{$\bullet$};
\draw (1,0) node{$\bullet$};
\draw (2,0) node{$\bullet$};

\end{tikzpicture}
\end{center}
\caption{$L_5^q$}
\label{L5}
\end{figure}

As $L_n^q$ is an iterated skew polynomial ring over $\K$, it has a PBW-basis $\{x_1^{a_1}x_2^{a_2}\dots x_n^{a_n}:(a_1,a_2,\dots,a_n)\in \N_0^n\}$, where $\N_0:=\N\cup\{0\}$. Provided $n\geq 2$, $L_n^q$ clearly satisfies the other conditions for a connected quantized Weyl algebra, the relevant graph for the presentation in
\eqref{Lpres1}-\eqref{Lpres3} being
the path graph $P_n$, and we may refer to it as the \emph{linear}
connected quantized Weyl algebra $L_n^q$. Figure~\ref{L5} shows the graph for $L_5^q$  as presented in \eqref{Lpres1}-\eqref{Lpres3}.
\end{example}

In the next example, the two ends of the graph are joined up and each vertex is linked to two others, resulting in the cycle graph, or $n$-gon, $C_n$.
\begin{example}\label{Cnq}
Let $n\geqslant 1$ be odd and let $q\in \K^*$. Let $C_n^q$ denote the $\K$-algebra generated by $x_1, x_2,\dots, x_n$ subject to the relations
\begin{alignat}{3} \label{Cpres1}
x_ix_{i+1}-qx_{i+1}x_i&= 1-q,&\quad& 1\leqslant i\leqslant n-1,\\
\label{Cpres2}
x_nx_1-qx_1x_n &= 1-q,&&\\
\label{Cpres3}
x_ix_j-qx_jx_i&=0, &\quad& i\geqslant 1,\; i+1<j\leqslant n,\; j-i\text{ odd},\\
\label{Cpres4}
x_ix_j-q^{-1}x_jx_i&=0, &\quad& i\geqslant 1,\; i+1<j< n,\; j-i\text{ even}.
\end{alignat}
In comparison with the odd case of $L_n^q$, the relation $x_1x_n-q^{-1}x_nx_1=0$ is replaced by the quantized Weyl relation
$x_nx_1-qx_1x_n = 1-q$.
As in Example~\ref{Lnq}, one can show that $C_n^q$ is the skew polynomial ring $L_{n-1}^q[x_n;\tau_n,\partial_n]$ for the same $\K$-automorphism $\tau_n$ of $L_{n-1}^q$ as for $L_n^q$
   and the $\tau_n$-derivation $\partial_n$ of $L_{n-1}^q$ such that
 \begin{alignat*}{2}
\partial_n(x_1)&=1-q,&\\
\partial_n(x_j)&=0, &\quad 2\leqslant j\leqslant n-2,\\
\partial_n(x_{n-1})&=1-q^{-1}.&
\end{alignat*}

As for $L_n^q$, the algebra  $C_n^q$  satisfies the conditions for a connected quantized Weyl algebra, the relevant graph for the presentation in \eqref{Cpres1}-\eqref{Cpres4} being
the cycle graph $C_n$. We may refer to it as the \emph{cyclic}
connected quantized Weyl algebra $C_n^q$. Figure~\ref{C5} shows the graph for $C_5^q$  as presented in \eqref{Cpres1}-\eqref{Cpres4}.

\begin{figure}[h]
\begin{center}
\begin{tikzpicture}[]
\begin{scope}[thick,decoration={
    markings,
    mark=at position 0.5 with {\arrow{>}}}
    ]
\draw[postaction={decorate}](90:1.4)--(18:1.4);
\draw[postaction={decorate}](18:1.4)--(306:1.4);
\draw[postaction={decorate}](306:1.4)--(234:1.4);
\draw[postaction={decorate}](234:1.4)--(162:1.4);
\draw[postaction={decorate}](162:1.4)--(90:1.4);
\end{scope}
\draw (90:1.7) node{$x_1$};
\draw (18:1.7) node{$x_2$};
\draw (306:1.7) node{$x_3$};
\draw (234:1.7) node{$x_4$};
\draw (162:1.7) node{$x_5$};

\draw (90:1.4) node{$\bullet$};
\draw (18:1.4) node{$\bullet$};
\draw (306:1.4) node{$\bullet$};
\draw (234:1.4) node{$\bullet$};
\draw (162:1.4) node{$\bullet$};

\draw (54:1.4) node{$q$};
\draw (126:1.4) node{$q$};
\draw (198:1.4) node{$q$};
\draw (270:1.4) node{$q$};
\draw (342:1.4) node{$q$};

\end{tikzpicture}
\end{center}
\caption{$C_5^q$}
\label{C5}
\end{figure}
\end{example}

\begin{rmk}
Cyclic connected quantized Weyl algebras with $n=3$ are related to the quantized enveloping algebra $U_q(sl_2)$. Provided $q\neq \pm 1$, localization of  $C_3^{q^2}$ at $\{x_1^i\}_{i\geq 1}$, which is a right and left Ore set by \cite[Lemma 1.4]{g}, gives $U_q(sl_2)$ in the equitable presentation \cite{teretal}.
\end{rmk}

\begin{example}\label{Weyl}
For a $\K$-algebra $R$ satisfying (i) and (ii) of Definition~\ref{cqwa}, neither the graph in (iii) nor its connectedness are invariants of the algebra. For example, consider the Weyl algebra $A_2$, generated by $x_1,x_2,y_1,y_2$ with relations $x_iy_i-y_ix_i=1$, $i=1,2$, and commutation relations for other pairs of generators. The graph for these generators has two connected components
each with two vertices and a single edge. However taking $x_1, y_1, x_1+x_2, y_1+y_2$ as generators gives a square and taking
$x_1, y_1, x_1+x_2, y_2$ gives the path graph $P_4$. Thus $A_2$ is a connected quantized Weyl algebra although this is not apparent from its usual presentation.

A similar situation exists for higher Weyl algebras $A_n$, $n\geq 3$, with an increasing variety of possible graphs for different sets of generators. For example, $A_3$ has the hexagonal graph $C_6$ for the generators $x_1,y_1,x_1+x_2,y_2,x_2+x_3,y_1+y_3$, the path graph $P_6$ for $x_1,y_1,x_2,y_2,x_2+x_3,y_1+y_3$
and the complete bipartite graph $K_{3,3}$ for $x_1,x_1+x_2,x_1+x_2+x_3,y_1,y_1+y_2,y_1+y_2+y_3$.

Whenever we refer to the graph for $L_n^q$ or $C_n^q$ we shall mean the graph for the presentation in \eqref{Lpres1}-\eqref{Lpres3} or \eqref{Cpres1}-\eqref{Cpres4} as appropriate.
\end{example}

\begin{example}\label{polyWeyl}
Consider the $\K$-algebra $R$ generated by $x_1, x_2$ and $x_3$ subject to the relations
\begin{eqnarray*}
x_1x_2-x_2x_1&=&1,\\
x_2x_3-x_3x_2&=&1,\\
x_3x_1-x_1x_3&=&1.
\end{eqnarray*}
This can be obtained from $C_3^q$ by first changing generators to replace the scalar terms $1-q$  in the relations by $1$ and then setting $q=1$.
Writing $x, y$ and $z$ for $x_1, x_2$ and $x_3$ respectively, $R$ is a skew polynomial ring $A_1[z;\delta]$, where $\delta$ is the derivation of the first Weyl algebra $A_1$ with $\delta(x)=1$ and $\delta(y)=-1$. As such it is a connected quantized Weyl algebra with the same graph as $C_3^q$. All derivations of $A_1$ are known to be inner \cite[Lemma 4.6.8]{dix} and $\delta$ is the inner derivation induced by $-(x+y)$. Setting $t=x+y+z$, which is central, $R$ is a polynomial ring $A_1[t]$. Relative to the generators $x, y, t$, the graph is not connected, having connected components $\{x,y\}$ and $\{t\}$.
A similar construction reveals the polynomial algebra $A_k[t]$ over the $k$th Weyl algebra $A_k$ to be a connected quantized Weyl algebra with the same cyclic graph as $C_{2k+1}^q$. We shall see that the algebra $C_{2k+1}^q$ of Example~\ref{Cnq} has a distinguished central element $\Omega$ which, loosely speaking, quantizes $t$
so that the quotient $C_{2k+1}^q/\Omega C_{2k+1}^q$ quantizes $A_k$.
\end{example}

The next result tells us that, for an algebra $R$ satisfying Condition (i) of Definition~\ref{cqwa}, Condition (ii) is independent of the ordering of the generators and that every algebra satisfying both is an iterated skew polynomial extension of $\K$.
\begin{lemma}\label{PBWindord}
Let $n\geq 2$ and, for $1\leq i\neq j\leqslant n$, let $q_{ij}\in \K^*$ and $r_{ij}\in \K$ be such that, for $i\neq j$, $q_{ji}=q_{ij}^{-1}$ and
  $r_{ji}=-q_{ij}^{-1}r_{ij}$. Let $R$ be the $\K$-algebra generated by $x_1,x_2,\dots,x_n$ subject to the $n(n-1)/2$ relations
\begin{equation}
x_jx_i-q_{ij}x_ix_j=r_{ij},\quad 1\leqslant i<j\leqslant n.\label{qijrij}
\end{equation}
The following are equivalent:
\begin{enumerate}
\item The standard monomials
 $x_1^{a_1}x_2^{a_2}\dots x_n^{a_n}$ form a $\K$-basis for $R$;
\item whenever $i, j, k$ are distinct, $r_{ij}\neq 0$ implies $q_{ik}=q_{kj}=q_{jk}^{-1}$;
\item $R$ is an iterated skew polynomial algebra over $\K$ in $x_{1}, x_{2},\dots,x_{n}$;
\item for any permutation $\sigma\in S_n$, the standard monomials
 $x_{\sigma(1)}^{a_1}x_{\sigma(2)}^{a_2}\dots x_{\sigma(n)}^{a_n}$ form a $\K$-basis for $R$;
\item for any permutation $\sigma\in S_n$, $R$ is an iterated skew polynomial algebra over $\K$ in $x_{\sigma(1)}, x_{\sigma(2)},\dots,x_{\sigma(n)}$.

\end{enumerate}
\end{lemma}
\begin{proof}  As (ii) is invariant under permutation and (iv) and (v) are obtained from (i) and (iii) respectively by permutation, it suffices to prove the equivalence of (i), (ii) and (iii).

(i)$\Rightarrow$(ii). Suppose that
 $\{x_1^{a_1}x_2^{a_2}\dots x_n^{a_n}\}$ is a $\K$-basis for $R$ and let $1\leq u<v<w\leq n$.  The element $x_wx_vx_u$ can be written in terms of this basis  by computing either $(x_wx_v)x_u$ or $x_w(x_vx_u)$. Equating the resulting two expressions and cancelling the term in $x_ux_vx_w$ that appears in both gives
 \[q_{vw}q_{uw}r_{uv}x_w+q_{vw}r_{uw}x_v+r_{vw}x_u=r_{uv}x_w+q_{uv}r_{uw}x_v+q_{uv}q_{uw}r_{vw}x_u.\]
If $r_{uv}\neq 0$ or, equivalently, $r_{vu}\neq 0$  then comparing coefficents of $x_w$ gives $q_{vw}=q_{wu}$.
If $r_{uw}\neq 0$, or, equivalently, $r_{wu}\neq 0$, then comparing coefficents of $x_v$ gives $q_{vw}=q_{uv}$.
If $r_{vw}\neq 0$, or, equivalently, $r_{wv}\neq 0$, then comparing coefficents of $x_u$ gives $q_{uv}=q_{wu}$.
This covers all six possibilities for the relative order of $i,j,k$ where $\{i,j,k\}=\{u,v,w\}$ in (ii).

(ii)$\Rightarrow$(iii). Suppose that (ii) holds. For $1\leqslant m\leqslant n$, let $R_m$ be the subalgebra of $R$ generated by $x_1, x_2,\dots,x_m$. Let $i<j<k$. Consider
the relation
\[x_jx_i-q_{ij}x_ix_j=r_{ij}\]
and let $x_i^\prime =q_{ik}x_i$ and $x_j^\prime=q_{jk}x_j$.  Then
\[x_j^\prime x_i^\prime-q_{ij}x_i^\prime x_j^\prime=r_{ij}.\] This is trivial if $r_{ij}=0$ and follows from (ii) otherwise.
This gives rise, inductively, to a $\K$-automorphism $\alpha_k$ of $R_{k-1}$ with $\alpha_k(x_i)=q_{ik}x_i$ for $1\leq i<k$.
To see that there is a $\alpha_k$-derivation $\delta_k$ of $R_{k-1}$ with $\delta_k(x_i)=r_{ik}$ for  $1\leq i<k$, we need to check that
\[r_{jk}x_i+\alpha(x_j)r_{ik}-q_{ij}r_{ik}x_j-q_{ij}\alpha(x_i)r_{jk}=0,\]
that is,
\[r_{jk}(1-q_{ij}q_{ik})x_i+r_{ik}(q_{jk}-q_{ij})x_j=0,\]
which is immediate from (ii).
Using \cite[Proposition 1]{fddaj}, it now follows, inductively, that $R_k=R_{k-1}[x_k;\alpha_k,\delta_k]$ and hence  that $R$ is an iterated skew polynomial algebra in $x_{1}, x_{2},\dots,x_{n}$ over $\K$.

(iii)$\Rightarrow$(i). This is immediate from the fact that, as a left $R$-module, any skew polynomial ring $R[x;\alpha,\delta]$ is free
as a left $R$-module with basis $1, x, x^2,\dots$.
\end{proof}

\begin{cor}\label{singleq}
Let $R$ be a connected quantized Weyl algebra generated by $x_1,x_2,\dots,x_n$ with parameters $q_{ij}$ and $r_{ij}$.
There exists $q\in \K^*$ such that
$\{q_{ij}:1\leq i\neq j\leq n\}=\{q,q^{-1}\}$.
\end{cor}
\begin{proof}
As $R$ is connected and every finite connected graph has a spanning tree, we can renumber the generators so that, for $1\leq m\leq n$, the subgraph corresponding to the subalgebra $R_m$ generated by $x_1,\dots,x_m$ is connected. By Lemma~\ref{PBWindord}, $R_m$ is a connected quantized Weyl algebra
for $m\geq 2$.
For $2\leq m\leq n$, let $Q_m=\{q_{ij}:1\leq i\neq j\leq m\}$. Then $Q_2=\{q,q^{-1}\}$ where $q=q_{12}$.
 Let $2\leq m< n$, suppose that $Q_m=\{q,q^{-1}\}$ and let $i=m+1$. There exist $j,k$ such that $1\leq k<j\leq m$, $r_{ij}\neq 0$ and
 $r_{jk}\neq 0$. By Lemma~\ref{PBWindord}(ii), $q_{ij}=q_{ki}=q_{jk}\in \{q,q^{-1}\}$ and $q_{i\ell}=q_{\ell j}\in
 \{q,q^{-1}\}$ for all $\ell\in\{1,2,\dots,m\}\backslash\{j\}$. It follows that $Q_{m+1}=\{q,q^{-1}\}$ and, by induction, that $Q_{n}=\{q,q^{-1}\}$ as required.
\end{proof}

\begin{cor}\label{nothree}
Let $R$ be a connected quantized Weyl algebra generated by $x_1,x_2,\dots,x_n$ with single parameter $q$ as in Corollary~\ref{singleq}.
Suppose that $q\neq\pm1$. If $r_{ij}\neq 0$ and $r_{jk}\neq 0$ then $r_{j\ell}=0$ for $\ell\in \{1,2,\dots,n\}\backslash \{i,j,k\}$. In other words, in the graph associated with the given presentation of $R$, the maximum degree of a vertex cannot exceed two.
\end{cor}
\begin{proof}
Suppose that $r_{ij}\neq 0$, $r_{jk}\neq 0$ and $r_{j\ell}\neq 0$. Without loss of generality, assume that $q_{ij}=q$. By Lemma~\ref{PBWindord}(ii) applied to $i, j$ and either $\ell$ or $k$, $q_{j\ell}=q_{jk}=q$ whereas, by Lemma~\ref{PBWindord}(ii) applied to $k, j$ and $\ell$, $q_{j\ell}=q_{kj}=q^{-1}$. As $q\neq\pm1$, this is a contradiction.
\end{proof}

The next  result is  readily checked  from the defining relations for $L_n^q$ and $C_n^q$.

\begin{prop}\label{someautos} Let $n\geq 1$ and let $q\in \K^*\backslash\{1\}$.
\begin{enumerate}
\item Let $\nu\in \K^*$. There is a $\K$-automorphism $\iota_\nu$ of $L_n^q$ such that $\iota_\nu(x_i)=\nu^{(-1)^i}x_i$ for $1\leq i\leq n$.
\item If $n$ is odd then there is a $\K$-automorphism of $\iota$ of $C_n^q$ such that $\iota(x_i)=-x_i$ for $1\leq i\leq n$.
\item There is an injective $\K$-homomorphism $\theta:L_{n-1}^q\rightarrow L_{n}^q$ such that $\theta(x_i)=x_{i+1}$ for $1\leq i\leq n-1$.
\item If $n$ is odd then there is a $\K$-automorphism $\theta$ of $C_n^q$ such that $\theta(x_i)=x_{i+1}$ for $1\leq i\leq n$, where subscripts are interpreted modulo $n$
in $\{1,2,\dots,n\}$.
\item There is a $\K$-isomorphism from  $L_n^q$ to $L_n^{q^{-1}}$ such that $x_i\mapsto x_{n-i+1}$ for $1\leq i\leq n$.
\item If $n$ is odd then there is a $\K$-isomorphism from  $C_n^q$ to $C_n^{q^{-1}}$ such that $x_i\mapsto x_{n-i+1}$ for $1\leq i\leq n$.
\end{enumerate}
\end{prop}

\begin{rmk}\label{restriction}
When $n$ is odd and $j<n$ then $L_{j-1}^q$ and $L_{j}^q$ are subalgebras of $C_n^q$ and the injective $\K$-homomorphism $\theta:L_{j-1}^q\rightarrow L_{j}^q$ in Proposition~\ref{someautos}(iii) is the restriction to $L_{j-1}^q$ of the automorphism $\theta$ of $C_n^q$ in Proposition~\ref{someautos}(iv).
\end{rmk}

\begin{prop}\label{LorC}
Let $n\geq 2$ and let $R$ be a connected quantized Weyl algebra generated by $x_1,x_2,\dots,x_n$ with single parameter $q$ as in Corollary~\ref{singleq}.
Suppose that $q\neq\pm1$. Then $R\simeq L_n^q$ or $R\simeq C_n^q$.
\end{prop}
\begin{proof}
  As in the proof of Corollary~\ref{singleq}, we can renumber the generators so that, for $2\leq i\leq n$,  the algebra $R_i$ generated by $x_1,\dots,x_i$ is a connected quantized Weyl algebra.
  In view of Proposition~\ref{someautos}(v,vi), we can assume that $q_{12}=q$. We shall show, by induction on $n$, that there exist $\mu_1,\dots,\mu_n\in \K^*$ such that if $x_i^\prime=\mu_i x_i$ then when the defining relations of $R$ are written in terms of the generators $x_i^\prime$ they become those of $L_n^q$ or $C_n^q$. This is true when $n=2$, take $\mu_1=1$ and $\mu_2=(1-q)r_{12}^{-1}$. By induction we may assume that  $R_{n-1}$ has a presentation for which the graph is the same as for $L_{n-1}^q$ or, if $n$ is even, $C_{n-1}^q$ and that, for $1\leq i<j\leq n-1$, $q_{ij}=q^{(-1)^{j-i+1}}$ and, for $1\leq i\leq n-2$,  $r_{i(i+1)}=1-q$.
     All vertices in the graph for $C_{n-1}^q$ and vertices $x_2,\dots,x_{n-1}$ in the graph representing $L_{n-1}^q$ have vertex degree $2$ so, by Corollary~\ref{nothree} and as $R$ is connected, the graph for $R_{n-1}$ must be the same as for $L_{n-1}^q$,  $r_{in}=0$
for $1<i<n-1$ and  $r_{n1}\neq 0$ or $r_{n-1,n}\neq 0$ or both.

     First suppose that $r_{(n-1)n}\neq 0$.
 As $r_{(n-1)n}\neq 0$, it follows from Lemma~\ref{PBWindord}(ii) that, for $1\leq i<n$,  $q_{in}=q_{i(n-1)}^{-1}=q^{(-1)^{n-i+1}}$. Thus   $q_{ij}=q^{(-1)^{j-i+1}}$ for $1\leq i<j\leq n$.
Now suppose also that $r_{n1}=0$. Let $x_i^\prime=x_i$ for $1\leq i\leq n-1$ and let $x_n^\prime=(1-q)r_{(n-1)n}^{-1}x_n$. Then $x_i^\prime,\dots, x_n^\prime$ generate $R$ subject to the $q$-commutation relations  $x_i^\prime x_j^\prime=q^{(-1)^{j-i+1}}x_jx_i$, $i<j-1$, and the quantized Weyl relations \[x_i^\prime x_{i+1}^\prime-q x_{i+1}^\prime x_i^\prime=1-q,\quad 1\leq i\leq n-1.\] Thus $R\simeq L_n^q$. Similar calculations with the generators ordered as $x_n,x_1,\dots,x_{n-1}$ give the same conclusion when $r_{n1}\neq 0$ and $r_{(n-1)n}=0$.

It remains to consider the case where $r_{1n}\neq 0\neq r_{(n-1)n}$. In this case,
by Lemma~\ref{PBWindord}(ii),  $q=q_{12}=q_{2n}=
q^{(-1)^{n-1}}$ so, as $q\neq \pm1$, $n$ must be odd.
The same change of generators as above gives the same relations between the $x_i^\prime$ but with the $q$-commutation relation between
$x_1^\prime$ and $x_n^\prime$ replaced by \[x_n^\prime x_1^\prime-qx_1^\prime x_n^\prime=(1-q)\lambda,\]
where $\lambda=r_{(n-1)n}^{-1}r_{n1}$. Let $\rho\in \K$ be such that $\rho^2=\lambda^{-1}$ and let
$x_i^{\prime\prime}=\rho^{(-1)^{i-1}}x_i^\prime$ for $1\leq i\leq n$.  Then $x_1^{\prime\prime},\dots, x_n^{\prime\prime}$ generate $R$ subject to the $q$-commutation relations
$x_i^{\prime\prime}x_j^{\prime\prime}=q^{(-1)^{j-i+1}}x_j^{\prime\prime}x_i^{\prime\prime}$, when $i<j-1$ unless $i=1$ and $j=n$, and the quantized Weyl relations
\[x_i^{\prime\prime} x_{i+1}^{\prime\prime}-q x_{i+1}^{\prime\prime}x_i^{\prime\prime}=\rho\rho^{-1}(1-q)=1-q,\quad 1\leq i\leq n-1,\]
and
\[x_n^{\prime\prime} x_1^{\prime\prime}-qx_1^{\prime\prime} x_n^{\prime\prime}=\rho^2(1-q)\lambda=1-q.\]
Thus $R\simeq C_n^q$ in this case.

\end{proof}

\begin{rmk}
Proposition~\ref{LorC} is false when $q=\pm 1$. If $q=\pm 1$ and  $q_{ij}$ is always $q$ when $i\neq j$, then, by Lemma~\ref{PBWindord}, one can take any connected graph $G$ on $x_1,x_2,\dots,x_n$ and construct a connected quantized Weyl algebra using the relations $x_ix_j-qx_jx_i=r_{ij}$, $1\leq i<j\leq n$,
where $r_{ij}=1$ if there is an edge between $x_i$ and $x_j$ in $G$ and  $r_{ij}=0$ if there is no such edge. In \cite{Wn} the authors consider the $\K$-algebra $W_n$ constructed in this way when $q=-1$ and $G$ is the complete graph.
\end{rmk}

Many ring theoretic properties of $C_n^q$ and $L_n^q$ follow from the fact that they are iterated skew polynomial extensions of the field $\K$.
\begin{prop}\label{ringprops}
\noindent \rm{(i)} The connected quantized Weyl algebras $C_n^q$ and $L_n^q$ are right and left noetherian  domains with $\K^*$ as their group of units.

\noindent \rm{(ii)} If $M$ is a simple module over either $C_n^q$ or $L_n^q$ then $\Endo M=\K$. If  $R$ is a prime factor ring of either $C_n^q$ or $L_n^q$
for which the centre $Z(R)$ of $R$ is not $\K$ then $R$ is not primitive.

\noindent \rm{(iii)} If  $R$ is a prime factor of either $C_n^q$ or $L_n^q$ then the Jacobson radical $\Jac(R)=0$. If, further, the intersection of the non-zero prime ideals of $R$ is non-zero then $R$ is primitive.
\end{prop}
\begin{proof}
\noindent \rm{(i)}.  $C_n^q$ and $L_n^q$ are right and left noetherian  domains by repeated application of \cite[Theorem 1.2.9(iv) and (i)]{McCR}
or \cite[Theorem 2.6 and Exercise 2O]{GW}.
Although explicit references are elusive, it is well-known and easy to see, using degree, that, for a skew polynomial ring $R[x;\alpha,\delta]$ over a ring $R$ with an automorphism $\alpha$ and $\alpha$-derivation $\delta$, $U(R[x;\alpha,\delta])=U(R)$. Hence   $U(C_n^q)=\K^*=U(L_n^q)$.

\noindent \rm{(ii)} and \rm{(iii)}. By \cite[Example 1.6.11]{McCR}, $C_n^q$ and $L_n^q$ are constructible $\K$-algebras in the sense of \cite[9.4.12]{McCR} so, by \cite[Theorem 9.4.21]{McCR},
they satisfy the Nullstellensatz over $\K$ as stated in \cite[9.1.4]{McCR}. Thus  every factor ring has nil Jacobson radical, and for any simple module $M$,  $\Endo M$ is algebraic over $\K$. Here $\K$ is algebraically closed so $\Endo M=\K$. If $Z(R)\neq \K$
and $\Phi\in Z(R)\backslash \K$ then multiplication by $\Phi$ induces an endomorphism of any simple module $M$. As $\Endo M=\K$, $M$ is annihilated by $\Phi-\mu$ for some $\mu\in \K$, whence $R$ cannot be primitive.

As $R$ is noetherian and prime, the nil ideal $\Jac(R)$ is nilpotent and hence $0$. If the intersection of the non-zero prime ideals of $R$ is non-zero then, as $\Jac(R)=0$, the ideal $0$ must be primitive.
\end{proof}

\begin{prop}\label{Lcompprime}
If $q$ is not a root of unity then  every prime ideal of $L_n^q$ is completely prime.
\end{prop}
\begin{proof}
This is a consequence of \cite[Theorem 2.3]{goodletzcompprime}. Conditions (a) and (b) of that result are clearly satisfied. Condition (c) holds because, in Example~\ref{Lnq}, $\tau_i\delta_i(x_j)=0=\delta_i\tau_i(x_j)$ for $1\leq j<i-1$ and $\tau_i\delta_i(x_{i-1})=1-q^{-1}=q\delta_i\tau_i(x_{i-1})$. Both (d) and the supplementary condition on the group $\Gamma$, which here is $\langle q\rangle$, hold because $q$ is not a root of unity.
\end{proof}

\begin{rmk}
We shall see that the analogue of Proposition \ref{Lcompprime} for $C_n^q$ is false. The conditions of \cite[Theorem 2.3]{goodletzcompprime} break down in a rather minimal way.
When the final generator $x_n$ is adjoined, Condition (c) fails because $\tau_n\partial_n(x_j)=0=\delta_n\partial_n(x_j)$ for $1<j<n-1$ and $\tau_n\partial_n(x_{n-1})=1-q^{-1}=q\partial_n\tau_n(x_{n-1})$, whereas $\tau_n\partial_n(x_{1})=1-q=q^{-1}\partial_n\tau_n(x_{1})$.
\end{rmk}

\section{Prime Spectrum of $L_n^q$}
The purpose of this section is to determine the prime spectrum of the linear connected quantized Weyl algebras $L_n^q$, $n\geq 3$.
It is well-known, for example \cite[8.4 and 8.5]{g} or \cite[Example 6.3(iii)]{ambiskew}, that the quantized Weyl algebra $A_1^q$ has a distinguished normal element  $z$ such that, if $q$ is not a root of unity, the localization of  $A_1^q$ at the powers of $z$ is simple. With $A_1^q$ as in \ref{Aoneq}, $z=x_1x_2-1=q(x_2x_1-1)$.
We shall identify a sequence of elements $z_i$, $-1\leq i\leq n$, such that, for $1\leq i\leq n$, $z_i$ is normal in $L_i^q$.

\begin{notn} Let $n\geq 3$. In $L^q_n$, let
$z_{-1} = 0$,  $z_0 = 1$ and, for $1\leq i\leq n$, let
$z_i = z_{i-1}x_i - z_{i-2}$.
\end{notn}

The next lemma gives an alternative expression for $z_i$ in terms of the $\K$-homomorphism $\theta:L_{n-1}^q\rightarrow L_{n}^q$ of Proposition~\ref{someautos}(iii), for which $\theta^2(L_{n-2}^q) \subset L_{n}^q$.

\begin{lemma}\label{altz} For $1\leq i\leq n$,
$z_i = x_1\theta(z_{i-1}) - \theta^2(z_{i-2})$.
\end{lemma}
\begin{proof} The proof is by induction on $i$.
When $i=1$, $x_1\theta(z_{0})-\theta^2(z_{-1}) = x_1 = z_1$
and, when $i=2$, $x_1\theta(z_1)-\theta^2(z_0) = x_1x_2 - 1 = z_2$.

If $i>2$ and the result holds for $i-1$ and $i-2$, then
\begin{eqnarray*}
z_i&=&z_{i-1}x_i - z_{i-2}\\
&=& x_1\theta(z_{i-2})x_i - \theta^2(z_{i-3})x_i - x_1\theta(z_{i-3}) + \theta^2 (z_{i-4})\\
&=&x_1\theta(z_{i-2}x_{i-1}-z_{i-3})-\theta^2(z_{i-3}x_{i-2}-z_{i-4})\\
&=& x_1\theta(z_{i-1})-\theta^2(z_{i-2}).
\end{eqnarray*}
\end{proof}

We now seek relations between $x_i$ and $z_j$, for $1\leq i,j \leq n$, and  between $z_i$ and $z_j$ when $i\neq j$.

\begin{lemma}\label{normalelementsthm}
Let $1\leq i,j \leq n$.
 Then
\[x_iz_j=\begin{cases}
q^{(-1)^{i-1}}z_jx_i&\text{ if $j$ is odd and }j<i-1,\\
z_jx_i&\text{ if $j$ is even and }j<i-1,\\
 z_{i-1}x_i+(q-1)z_{i-2}&\text{ if $i$ is odd and }j=i-1,\\
q^{-1}z_{i-1}x_i+(1-q^{-1})z_{i-2}&\text{ if $i$ is even and }j=i-1,\\
z_jx_i&\text{ if $j$ is odd  and }j\geq i,\\
q^{(-1)^{i-1}}z_jx_i&\text{ if $j$ is even and }j\geq i.
\end{cases}\]
\end{lemma}

\begin{proof}
This is a straightforward induction on $j$ using the defining relations and the  equations $z_k=z_{k-1}x_k-z_{k-2}$, and applying the previous two cases at each inductive step. The trivial case $j=0$, where $z_j=1$, can be used along with the case $j=1$ in the initial step. Separate calculations are needed for the cases $j<i-2$, $j=i-1$, $j=i$, $j=i+1$ and $j\geq i+2$. We give details for the cases $j=i-1$, $j=i$ and $j=i+1$ when $j$ is even. The odd cases of these are similar and the cases $j<i-2$ and $j\geq i+2$ are routine.
Suppose that $j$ is even and that the result holds for $j-1$ and $j-2$.
If $j=i-1$  then
\begin{eqnarray*}
x_iz_{i-1}&=&x_iz_{i-2}x_{i-1}-x_iz_{i-3}\\
&=&z_{i-2}(x_{i-1}x_i-(1-q))-z_{i-3}x_i\\
&=&z_{i-1}x_i-(1-q)z_{i-2}.
\end{eqnarray*}

If $j=i$  then
\begin{eqnarray*}
x_iz_i&=&x_iz_{i-1}x_{i}-x_iz_{i-2}\\
&=&q^{-1}z_{i-1}x_i^2+(1-q^{-1})z_{i-2}x_i-z_{i-2}x_i\\
&=&q^{-1}z_{i}x_i
\end{eqnarray*}
and if
$j=i+1$ then
\begin{eqnarray*}
x_iz_{i+1}&=&x_iz_{i}x_{i+1}-x_iz_{i-1}\\
&=&qz_i(x_{i+1}x_i+(1-q))-z_{i-1}x_i+(1-q)z_{i-2}\\
&=&qz_ix_{i+1}x_i+(1-q)z_{i-1}x_i-z_{i-1}x_i\\
&=&q(z_ix_{i+1}x_i-z_{i-1})=qz_{i+1}x_i.
\end{eqnarray*}

\end{proof}

\begin{cor}\label{normalelementsthm2}
For $2\leq i\leq n$,  \[x_iz_{i-1}=
\begin{cases}
qz_{i-1}x_i+(1-q)z_{i}&\text{ if $i$ is odd},\\
z_{i-1}x_i+(q^{-1}-1)z_{i}&\text{ if $i$ is even}.
\end{cases}\]
\end{cor}
\begin{proof} By Lemma~\ref{normalelementsthm},
 \[x_iz_{i-1}=
 \begin{cases}
z_{i-1}x_i+(q-1)z_{i-2}&\text{ if $i$ is odd},\\
q^{-1}z_{i-1}x_i+(1-q^{-1})z_{i-2}&\text{ if $i$ is even}.
\end{cases}\]
The result follows by using the equation $z_{i-2}=z_{i-1}x_i-z_i$ to substitute for $z_{i-2}$ on the right hand side.
\end{proof}

\begin{cor}\label{znnormal}
\begin{enumerate} \item
For $1\leq i\leq n$, $z_nx_i=\rho_i x_iz_n$, where
$\rho_i=
1$ if $n$ is odd and $\rho_i=
q^{(-1)^i}$ if $n$ is even. Consequently, the element $z_n$ is normal in $L_n^q$.

\item For $0\leq i<j\leq n$,
 $z_iz_j=q^{\lambda_{ij}} z_jz_i$, where
$\lambda_{ij}=
0$ if $j$ is odd  or if $j$ and $i$ are both even, and
$\lambda_{ij}=1$ if $j$ is even and $i$ is odd.
\end{enumerate}
\end{cor}

\begin{proof}
(i) is immediate from Lemma~\ref{normalelementsthm} and (ii) follows inductively using the formula $z_i=z_{i-1}x_i-z_{i-2}$.
\end{proof}

\begin{notn}\label{WnTn} Let $n\geq 2$, let $q\in \K^*$, let $\Lambda=(\lambda_{ij})$ be the $n\times n$ antisymmetric matrix over $\K$ such that, for $1\leq i<j\leq n$,
$\lambda_{ij}$ is as specified in Lemma~\ref{znnormal}(ii) and, for $1\leq i,j\leq n$, let $q_{ij}=q^{\lambda_{ij}}$. Thus $q_{ji}=q_{ij}^{-1}$, $q_{ii}=1$ and, for $i<j$, $q_{ij}=
1$ if $j$ is odd  or if $j$ and $i$ are both even, and
$q_{ij}=q$ if $j$ is even and $i$ is odd.

Let $W_n^q$ and $T_n^q$, respectively, denote the co-ordinate ring of quantum $n$-space with generators $z_1,\dots,z_n$ and the co-ordinate ring of the quantum $n$-torus with generators $z_1^{\pm 1},\dots,z_n^{\pm 1}$, subject to the relations
$z_iz_j=q_{ij}z_jz_i$ for $1\leq i,j\leq n$.
We may also refer occasionally to the subalgebras $W_i^q$ and $T_i^q$ generated by $z_1,z_2,\dots,z_i$ or by $z_1^{\pm 1},z_2^{\pm 1},\dots,z_i^{\pm 1}$ as appropriate, where $1<i<n$.
As is well-known, $W_n^q$ is an iterated skew polynomial algebra over $\K$, and $T_n^q$ is an  iterated skew Laurent polynomial algebra
over $\K$.\end{notn}

\begin{lemma}\label{Tsimpleorpolysimple} Suppose that $q$ is not a root of unity.
If $n$ is even then $T_n^q$ is simple and if $n$ is odd then $T_n^q$ is a Laurent polynomial ring $T_{n-1}^q[z_n^{\pm 1}]$ over the simple ring
$T_{n-1}^q$.
\end{lemma}
\begin{proof} If $n$ is even we apply the criterion given by \cite[Proposition 1.3]{McCP}. Let $m_1,\dots,m_n\in \Z$ be such that $q_{1j}^{m_1}q_{2j}^{m_2}\dots q_{nj}^{m_n}=1$ for all $j$, $1\leq j\leq n$. If $j$ is even then $q_{ij}=1$ unless $i$ is odd and $i<j$, in which case $q_{ij}=q$. So successive considerations of the cases $j=2, 4,\dots, n$ gives $m_2=m_4=\dots=m_n=0$. If $j$ is odd  then $q_{ij}=1$ unless $i$ is even and $i>j$,  in which case $q_{ij}=q^{-1}$. Successive considerations of the cases $j=n-1, \dots, 3, 1$ gives $m_{n-1}=\dots=m_3=m_1=0$. By \cite[Proposition 1.3]{McCP}, $T_n^q$ is simple. If $n$ is odd then it is clear from Corollary~\ref{znnormal}(ii) that $T_n^q=T_{n-1}^q[z_n^{\pm 1}]$.
\end{proof}

The following result of Wexler-Kreindler \cite[Proposition 2]{wexler} on changing the indeterminate of a skew polynomial ring will be useful.
\begin{lemma}\label{sprgenchange}
Let $\alpha$ be an automorphism of a ring $R$ and let $\delta$ be an $\alpha$-derivation of $R$. Let $a\in R$ and let $u$ be a unit in $R$ with inner
automorphism $\gamma_u:r\mapsto uru^{-1}$. Let $\alpha^\prime=\gamma_u\alpha$ and, for $r\in R$, let $\delta^\prime(r)=u\delta(r)+ar-\gamma_u(\alpha(r))a$. Then
$\delta^\prime$ is an $\alpha^\prime$-derivation of $R$ and $R[x;\alpha,\delta]=R[x^\prime;\alpha^\prime,\delta^\prime]$, where $x^\prime=ux+a$.
\end{lemma}

The next result allows us to identify $L_n^q$ with an intermediate $\K$-algebra between $W_n^q$ and $T_n^q$. After the proof of (i) this identification, which is pre-empted by the use of the notation $z_i$ in all three algebras, will be made implicitly.
\begin{notn} For $1\leq j\leq n$, let $\mathcal{Z}_j$ denote the multiplicatively closed set $\{fz_1^{a_1}z_2^{a_2}\dots z_j^{a_j}: f\in \K^*, a_i\in \N_0 , 1\leq i\leq j\}$ of $W_n^q$. We will make use of the fact that each $\mathcal{Z}_n$ is right and left Ore and that $T_n^q$ is the localization of  $W_n^q$ at $\mathcal{Z}_n$.
\end{notn}

\begin{prop}\label{n-1loc}
Let $n\geq 2$ and $q\in \K^*$. Suppose that $q$ is not a root of unity.
\begin{enumerate}
\item There are injective $\K$-algebra homomorphisms $\phi:W_n^q\hookrightarrow L_n^q$ and $\psi:L_n^q\hookrightarrow T_n^q$ such that $\psi\phi(z_i)=z_i$ for $1\leq i\leq n$. (This allows us to regard $W_n^q$ as a subalgebra of $L_n^q$ and to regard each $\mathcal{Z}_j$ as a subset of $L_n^q$.)
\item The sets $\mathcal{Z}_n$ and $\mathcal{Z}_{n-1}$ are right and left Ore sets in $L_n^q$. The localization of $L_n^q$ at $\mathcal{Z}_n$ is $T_n^q$. If $n$ is odd then the localization of $L_n^q$ at  $\mathcal{Z}_{n-1}$
    is the polynomial algebra $T_{n-1}^q[z_n]$ and if $n$ is even it is the skew polynomial algebra $T_{n-1}^q[z_n;\alpha]$, where
    $\alpha(z_i)=z_i$ if $i$ is even and $\alpha(z_i)=q^{-1}z_i$ if $i$ is odd.
\end{enumerate}
\end{prop}
\begin{proof}

 (i) By Corollary~\ref{znnormal}(ii), applied to each of the algebras $L_i^q$, there is a $\K$-algebra homomorphism $\phi:W_n^q\rightarrow L_n^q$ such that $\phi(z_i)=z_i$ for $1\leq i\leq n$. As $L_n^q$ is a domain so also is the image $\phi(W_n^q)$. Hence $\ker \phi$ is a completely prime ideal of $W_n^q$. If $\ker \phi\neq 0$ then, by Lemma~\ref{Tsimpleorpolysimple}, either $z_i\in \ker \phi$ for some $i$, $1\leq i\leq n-1$ or $z_n-\lambda\in \ker\phi$ for some $\lambda \in \K$ (with $\lambda=0$ if $n$ is even). As $L_n^q$ has a PBW basis and the coefficient of $x_1x_2\dots x_i$ in $z_i$ is $1$, it follows that $\ker \phi=0$ and hence that $\phi$ is injective.

With $z_{-1}=0$ and $z_0=1$, let $w_i=z_{i-1}^{-1}z_i$ and $v_i=w_i+w_{i-1}^{-1}=z_{i-1}^{-1}(z_i+z_{i-2})\in T_n^q$, $1\leq i\leq n$.
It follows, by Corollary~\ref{znnormal}(ii), that, for $1\leq i<j\leq n$, \[w_jw_i=\begin{cases}q^{-1}w_iw_j& \text{ if $i-j$ is even,}\\
qw_jw_i& \text{ if $i-j$ is odd.}\\
\end{cases}\]
From this it follows routinely that if $i>j+1$ then
\[v_jv_i=\begin{cases}q^{-1}v_jv_i& \text{ if $i-j$ is even,}\\
qv_jv_i& \text{ if $i-j$ is odd.}\\
\end{cases}\]
and that
\[v_{i+1}v_i=qv_iv_{i+1}+1-q.\]
Thus the $v_i$'s satisfy the defining relations for $L_n^q$ and there is  a $\K$-algebra homomorphism $\psi:L_n^q\rightarrow T_n^q$ such that $\psi(x_i)=v_i$ and $\psi(z_i)=z_i$ for $1\leq i\leq n$. As $T_n^q$ is a domain so also is $\psi(L_n^q)$ so $\ker \psi$ is a  completely prime ideal of $L_n^q$.

Let $D=\{r\in L_n^q: zr\in W_n^q \text{ for some }z\in \mathcal{Z}_{n-1}\}$. Using Corollary~\ref{znnormal}(ii), it is easy to see that $D$ is a subalgebra of $L_n^q$. For $1\leq i\leq n$, $z_{i-1}x_i=z_i+z_{i-2}$ so each $x_i\in D$ and therefore $L_n^q=D$. Now suppose that $\ker \psi\neq 0$ and let $0\neq f\in \ker \psi$. As $f\in D$, there exists $z\in \mathcal{Z}_{n-1}$ such that $zf\in \ker \psi\cap W_n^q$. As $\psi(z_i)=z_i$ for all $i$,
$\ker \psi\cap W_n^q=0$ so $zf=0$. But $L_n^q$ is a domain so $f=0$.
Hence $\ker \psi=0$ and $\psi$ is injective.

(ii) As $T_n^q$ is a right and left quotient ring of $W_n^q$ with respect to $\mathcal{Z}_n$, it is also a right and left quotient ring of the intermediate ring $L_n^q$ with respect to $\mathcal{Z}_n$ and hence $\mathcal{Z}_n$ is right and left Ore in $L_n^q$. The set $\mathcal{Z}_{n-1}$ is a $\tau_n$-invariant right and left Ore set in $L_{n-1}^q$, with quotient ring
$T_{n-1}^q$, so it follows easily from \cite[Lemma 1.4]{g} that it is right and left Ore in $L_n^{q}$ with quotient ring
$T_{n-1}^q[x_n;\tau_n,\delta_n]$. As $z_n=z_{n-1}x_n-z_{n-2}$ and $z_{n-1}$ is invertible in $T_{n-1}^q$, it follows from
Lemma~\ref{sprgenchange} that if $n$ is odd then $T_{n-1}^q[x_n;\tau_n,\delta_n]=T_{n-1}^q[z_n]$ and that if $n$ is even it is the skew polynomial algebra $T_{n-1}^q[x_n;\tau_n,\delta_n]=T_{n-1}^q[z_n;\tau]$ where
    $\tau(z_i)=z_i$ if $i$ is even and $\tau(z_i)=q^{-1}z_i$ if $i$ is odd.
\end{proof}

\begin{notn}
For $m\in \N$ and $q\in \K^*$, let $\qnum{m}{q}:=1+q+q^2+\ldots+q^{m-1}$, which, if $q\neq 1$, is $(q^m-1)/(q-1)$.
\end{notn}

\begin{prop}\label{LzL}
Let $n\geq 2$ and let $z=fz_1^{a_1}z_2^{a_2}\dots z_{n-1}^{a_{n-1}}\in \mathcal{Z}_{n-1}$, where $f\in \K^*$ and  $a_i\in \N_0$ for $1\leq i\leq n-1$. Let $L=L_n^q$ and suppose that $q$ is not a root of unity. Then $LzL=L$ and $P\cap \mathcal{Z}_{n-1}=\emptyset$ for all prime ideals $P$ of $L$.
\end{prop}

\begin{proof}
Suppose that $LzL\neq L$ and let $M$ be a maximal ideal of $L$ containing $LzL$. By Proposition~\ref{Lcompprime}, $M$ is completely prime and so $z_i\in M$ for some $i$, $1\leq i\leq n-1$. By Lemma~\ref{normalelementsthm}, $z_j\in M$ for $0\leq j\leq i-1$. In particular, $1=z_0\in M$ which is impossible. Hence $LzL=L$ and consequently $P\cap \mathcal{Z}_{n-1}=\emptyset$ for all prime ideals $P$ of $L$.
\end{proof}

We are now in a position to determine the prime ideals of $L_n^q$ when $q$ is not a root of unity.

\begin{theorem}\label{specL} Let $L=L_n^q$ and suppose that $q$ is not a root of unity.
If $n$ is odd then the prime ideals of $L$ are $0$ and, for each $\lambda\in \K$, $P_\lambda:=(z_n-\lambda)L$.
If $n$ is even then the prime ideals of $L$ are $0$, $z_nL$  and, for each $\lambda\in \K^*$, $P^\prime_\lambda:=(z_{n-1}-\lambda)L+z_nL$.
\end{theorem}
\begin{proof} First note that if $n$ is odd then $P_\lambda$ is an ideal of $L$ as $z_n$ is central. If $n$ is even, then $z_nL$ is an ideal of $L$, by Corollary~\ref{znnormal}(i), and $z_{n-1}$ is central modulo $z_nL_n^q$ in $L_n^q$, by Corollary~\ref{normalelementsthm2}. So $P^\prime_\lambda$ is an ideal of $L$.

For $\lambda\in \K$, note that $z_n-\lambda=z_{n-1}x_n-z_{n-2}-\lambda$ has degree one in $x_n$. Using \cite[Proposition 1]{normalels} and Corollary~\ref{znnormal}, it is easily shown by induction that $(z_n-\lambda)L$ is completely prime if $n$ is odd or if $n$ is even and $\lambda=0$.  If $\lambda\neq 0$, $x_n\equiv \lambda^{-1}z_{n-2}\bmod P^\prime_\lambda$, so $L/P^\prime_\lambda$ is generated by the images $\overline{x_i}$, $1\leq i\leq n-1$, and there are inverse homomorphisms $\varphi:L/P^\prime_\lambda\rightarrow L_{n-1}^q/(z_{n-1}-\lambda)L_{n-1}^q$ and $\theta:L_{n-1}^q/(z_{n-1}-\lambda)L_{n-1}^q\rightarrow L/P^\prime_\lambda$  such that $\varphi(\overline{x_i})=\overline{x_i}$ and $\theta(\overline{x_i})=\overline{x_i}$ for $1\leq i\leq n-1$. Thus $L/P^\prime_\lambda\simeq L_{n-1}^q/(z_{n-1}-\lambda)L_{n-1}^q$ and hence $P^\prime_\lambda$ is also completely prime.

Let $T$ be the localization of $L$ at $\mathcal{Z}_{n-1}$ and let $P$ be a non-zero prime ideal of $L$. Suppose that  $n$ is odd. By Proposition~\ref{n-1loc} and Lemma~\ref{Tsimpleorpolysimple},
$T$ is the polynomial ring $T_{n-1}^q[z_n]\simeq T_{n-1}^q\otimes_\K \K[z_n]$ over the simple ring $T_{n-1}^q$. By \cite[Proposition 1.3]{McCP}, the centre of $T_{n-1}^q$ is $\K$. It is then a consequence of \cite[Lemma 9.6.9(i)]{McCR} that  $PT=P_\lambda T$ for some $\lambda\in \K$. By Proposition~\ref{LzL} and \cite[Proposition 2.1.16(vii)]{McCR}, $P=P_\lambda$.

Now suppose that $n$ is even. By Proposition~\ref{n-1loc} and Lemma~\ref{Tsimpleorpolysimple},
$T$ is a skew polynomial ring $T_{n-1}^q[z_n;\alpha]$ and $T_{n-1}^q[z_n^{\pm 1};\alpha]$ is simple. It follows that $PT$ contains $z_n$ and hence has the form $Q+z_nT$ for some prime ideal $Q$ of $T_{n-1}^q=T_{n-2}^q[z_{n-1}^{\pm 1}]$. By \cite[Lemma 9.6.9(i)]{McCR}, every such prime ideal of $T_{n-1}^q$ has the form $0$ or $(z_{n-1}-\lambda)T_{n-1}^q$ for some $\lambda\in \K^*$ so $PT=z_nT$ or $PT=P^\prime_\lambda T$. The result in the even case now follows as in the odd case using \cite[Proposition 2.16(vii)]{McCR} and Proposition~\ref{LzL}.

\end{proof}

Recall from \cite{Chatters} that a noetherian domain R is called a  unique factorization domain (UFD) if every non-zero prime ideal
of R contains a non-zero completely prime ideal of the form $pR$ where $p$ is normal in $R$. The following is immediate from
Theorem~\ref{specL}.

\begin{cor} Suppose that $q$ is not a root of unity.
If $n\geq 2$ then the linear connected quantized Weyl algebra $L_n^q$ is a UFD. If $n$ is odd then every height one prime ideal of $L_n^q$ is maximal.
If $n$ is even  then there is a unique height one prime ideal  $z_nL_n^q$ and $L_n^q/z_nL_n^q$ is a UFD in which every height one prime ideal is maximal.
\end{cor}

\begin{cor}\label{primitivespecL} Let $L=L_n^q$ and suppose that $q$ is not a root of unity.
If $n$ is odd then the primitive ideals of $L$ are the prime ideals $(z_n-\lambda)L$, $\lambda\in \K$.
If $n$ is even then the primitive ideals of $L$ are $0$ and, for each $\lambda\in \K^*$, $(z_{n-1}-\lambda)L+z_nL$.
\end{cor}
\begin{proof}
Suppose that $n$ is odd. The ideal $0$ of $L$ is not primitive, by Proposition~\ref{ringprops}(ii), due to the existence of the central element $z_n$. All the non-zero prime ideals are maximal and hence primitive.

Now suppose that $n$ is even. The ideal $0$ of $L$ is primitive, by Proposition~\ref{ringprops}(iii), because $L$ has a unique height one prime ideal. The ideal $z_nL$ is not primitive, by Proposition~\ref{ringprops}(ii), as $z_{n-1}$ is central modulo $z_nL$. The remaining prime ideals are maximal and hence primitive.
\end{proof}

\section{Prime Spectrum of $C_n^q$}
Throughout this section, $n\geq 3$ is an odd positive integer. Recall that the cyclic connected quantized Weyl algebra $C_n^q$ has the form $L_{n-1}^q[x_n;\tau_n,\partial_n]$ and contains
 $L_m^q$ as a subalgebra for $1\leq m\leq n-1$.  Let $\theta$ be the $\K$-automorphism of $C_n^q$ specified in Proposition~\ref{someautos}(iii). Thus $\theta(x_i)=x_{i+1}$ for $1\leq i\leq n-1$ and $\theta(x_n)=x_{1}$. For $1\leq m<n$, $\theta(L_m^q)$ is the $\K$-subalgebra of $C_n^q$ generated by
 $x_2,x_3,\dots,x_{m+1}$.
The relations that the  elements
$z_1,z_2,\dots,z_{n-1}\in C_n^q$ satisfy with each other, and with $x_1,x_2,\dots,x_{n-1}$, are as in Lemma~\ref{normalelementsthm} and Corollaries \ref{normalelementsthm2} and \ref{znnormal}. The next lemma gives the relations between $x_n$ and each of $z_1,z_2,\dots,z_{n-1}$   in $C_n^q$ and, by modifying the formula that defined $z_n$ in $L_n^q$, identifies a distinguished central element of $C_n^q$.

\begin{lemma}\label{xzC}
Let $\Omega=z_{n-1}x_n - z_{n-2} - q\theta(z_{n-2})\in C_n^q$ and let $\theta$ be the $\K$-automorphism of $C_n^q$ specified in Proposition~\ref{someautos}(iii).
\begin{enumerate}
\item For $1\leq j\leq n-2$,
\[x_nz_j=\begin{cases}
qz_jx_n+(1-q)\theta(z_{j-1})&\text{ if }j\text{ is odd},\\
z_jx_n+(1-q)\theta(z_{j-1})&\text{ if }j\text{ is even}.
\end{cases}\]
\item $x_nz_{n-1}=z_{n-1}x_n+(1-q)(\theta(z_{n-2})-z_{n-2})$.
\item $\theta(\Omega)=\Omega$.
\item $\Omega$ is central in $C_n^q$.
\end{enumerate}

\end{lemma}
\begin{proof}
 Recall that $C_n^q=L_{n-1}^q[x_n;\tau_n,\partial_n]$ for the $\K$-automorphism $\tau_n$ of $L_{n-1}^q$ such that $\tau_n(x_j)=q^{(-1)^{j-1}}x_j$ for $1\leq j\leq n-1$ and the $\tau_n$-derivation $\partial_n$ of $L_{n-1}^q$
such that
$\partial_n(x_1)=1-q$, $\partial_n(x_j)=0$ for $2\leq j\leq n-2$ and
$\partial_n(x_{n-1})=1-q^{-1}.$

For (i), let $1\leq j\leq n-2$. It is a routine matter to check, by induction, that
$\tau_n(z_j)=qz_j$ if $j$  is odd and $\tau_n(z_j)=z_j$ if $j$ is even
and that $\partial_n(z_j)=\partial_n(z_{j-1}x_j-z_{j-2})=(1-q)\theta(z_{j-1})$.
The result (i) follows.

Using (i), we see that $\tau_n(z_{n-1})=\tau_n(z_{n-2}x_{n-1}-z_{n-3})=z_{n-1}$ and that
\begin{eqnarray*}
\partial_n(z_{n-1})&=&\partial_n(z_{n-2}x_{n-1}-z_{n-3})\\
&=&(1-q)\theta(z_{n-3})\theta(x_{n-2})+(1-q^{-1})qz_{n-2}-(1-q)\theta(z_{n-4})\\
&=&(1-q)(\theta(z_{n-2})-z_{n-2}).
\end{eqnarray*}
Thus (ii) holds. For (iii),
\begin{alignat*}{2}
\theta(\Omega)&=\theta(z_{n-1}x_n) - \theta(z_{n-2}) - q\theta^2(z_{n-2})&&\\
&=\theta(x_nz_{n-1} + (1-q)(z_{n-2} - \theta(z_{n-2}))) - \theta(z_{n-2}) - q\theta^2(z_{n-2})&&\text{ (by (ii))}\\
&=x_1\theta(z_{n-1}) - q\theta(z_{n-2}) - \theta^2(z_{n-2})&&\\
&=x_1\theta(z_{n-2})x_n - x_1\theta(z_{n-3}) - \theta^2(z_{n-3})x_n + \theta^2 (z_{n-4}) - q\theta(z_{n-2})&&\\
&=z_{n-1}x_n - z_{n-2} - q\theta(z_{n-2})&&\text{ (by \ref{altz})}\\
&=\Omega.&&
\end{alignat*}

For (iv), first note that, by Lemma~\ref{normalelementsthm}, $x_{n-1}z_{n-2}=q^{-1}z_{n-2}x_{n-1}+(1-q^{-1})z_{n-3}$
so that
\begin{equation}
qx_n\theta(z_{n-2})=\theta(qx_{n-1}z_{n-2})=\theta(z_{n-2}x_{n-1}+(q-1)z_{n-3}).\label{xnthetaznminus2}
\end{equation}
By (ii), (i) and \eqref{xnthetaznminus2},
\begin{alignat*}{2}
x_n\Omega&=x_n(z_{n-1}x_n - z_{n-2} - q\theta(z_{n-2}))&&\\
&=(z_{n-1}x_n+(1-q)(\theta(z_{n-2})-z_{n-2}))x_n-qz_{n-2}x_n-(1-q)\theta(z_{n-3})-qx_n\theta(z_{n-2})&&\\
&=z_{n-1}x_n^2+(1-q)\theta(z_{n-2})x_n-z_{n-2}x_n-(1-q)\theta(z_{n-3})-qx_n\theta(z_{n-2})&&\\
&=z_{n-1}x_n^2-z_{n-2}x_n-q\theta(z_{n-2})x_n&&\\
&=\Omega x_n.&&
\end{alignat*}
By (iii),  $x_i\Omega=\Omega x_i$  for $1\leq i\leq n$, so  $\Omega$   is central in $C_n^q$.
\end{proof}

\begin{prop}\label{Cloc}
Let $q\in \K^*$. Suppose that $q$ is not a root of unity.
 The subsets $\mathcal{Z}_{n-1}$ and $\mathcal{Z}_{n-2}$ of $L_{n-1}^q$ are right and left Ore sets in $C_n^q$. The localization of $C_n^q$ at $\mathcal{Z}_{n-1}$
    is the polynomial ring $T_{n-1}^q[\Omega]$ over the simple algebra $T_{n-1}^q$.
\end{prop}
\begin{proof}
By Proposition~\ref{n-1loc}, $\mathcal{Z}_{n-1}$ is a right and left Ore set in $L_{n-1}^q$ and it is clearly $\tau_n$-invariant. By \cite[Lemma 1.4]{g}, $\mathcal{Z}_{n-1}$  is right and left Ore in $C_n^{q}$ with quotient ring
$T_{n-1}^q[x_n;\tau_n,\partial_n]$. Similarly, $\mathcal{Z}_{n-2}$ is a right and left Ore set in $L_{n-1}^q$ and in $C_n^{q}$.

As $\Omega=ux_n+a$, where $u=z_{n-1}$, which is invertible in $T_{n-1}^q$, and $a=-(z_{n-2}+q\theta(z_{n-2}))\in L_{n-1}^q$, it follows from Lemma~\ref{sprgenchange} that $T_{n-1}^q[x_n;\tau_n,\partial_n]$ has the form
$T_{n-1}^q[\Omega;\tau_n^\prime,\partial_n^\prime]$. As $\Omega$ is central, $\tau_n^\prime$ must be the identity automorphism on $L_{n-1}^q$ and $\partial_n^\prime$ must be the zero derivation.
\end{proof}

The next lemma will be significant in identifying the localization of $C_n^q$ at $\mathcal{Z}_{n-2}$ as an  \emph{ambiskew polynomial ring}.
Our notation for such rings  and the related \emph{generalized Weyl algebras} will be essentially as in \cite{ambiskew}.
Given a $\K$-algebra $A$, commuting $\K$-automorphisms $\alpha$ and $\gamma$ of $A$, an element $v\in A$ such that $va=\gamma(a)v$ for all $a\in A$  and $\gamma(v)=v$, and a scalar
$\rho\in
\K\backslash\{0\}$,
 the ambiskew polynomial ring $R=R(A,\alpha,v,\rho)$ is the iterated skew polynomial ring
$A[y;\alpha][x;\beta,\delta]$, where $\beta=\alpha^{-1}\gamma$ is extended to a $\K$-automorphism of
$A[y;\alpha]$ by setting $\beta(y)=\rho y$ and $\delta$ is a
$\beta$-derivation of $A[y;\alpha]$ such that $\delta(A)=0$ and
$\delta(y)=v$. Thus $xy=\rho yx+v$ and, for all $a\in A$,
$ya=\alpha(a)y$ and
$xa=\beta(a)x$.

If $v$ is regular, as will be the case in all examples considered here, then $v$ determines $\gamma$.

If there exists $u\in A$ such that $ua=\gamma(a)u$ for all $a\in A$ and
$v=u-\rho\alpha(u)$ then the element
$z:=xy-u=\rho(yx-\alpha(u))$ is such that
$zy=\rho yz$, $zx=\rho^{-1} xz$, $za=\gamma(a)z$ for all $a\in A$ and $zu=uz$.
If such an element $u$ exists then
$R$ is a {\it conformal ambiskew polynomial ring}, $u$ is a {\it splitting element} and $z$, which is normal in $R$, is the corresponding {\it Casimir element}
of $R$. The factor $R/zR$, which we denote here by $W(A,\alpha,u)$, is then generated by $A$, $X:=x+zR$ and $Y:=y+zR$ subject to the relations
$XY=u$, $YX=\alpha(u)$ and, for all $a\in A$,
$Ya=\alpha(a)Y$ and $Xa=\beta(a)X$.
In the case where $u$ is central this is one of the algebras named \emph{generalized Weyl algebras} in \cite{vlad1} and we use the same name here.

\begin{lemma}\label{w}
For $q\in \K^*$ and $1\leq i\leq n-1$,
\begin{alignat*}{3}
qz_i\theta(z_{i-2})-z_{i-1}\theta(z_{i-1})&=&\;-q^\frac{i-1}{2}&\quad \text{ if $i$ is odd and}\\
z_i\theta(z_{i-2})-z_{i-1}\theta(z_{i-1})&=&\;-q^\frac{i-2}{2}& \quad \text{ if $i$ is even}.
\end{alignat*}
\end{lemma}
\begin{proof}
When $i=1$, $z_{i-2}=0$ and $z_{i-1}=1$ so the result holds. Let $i>1$ and suppose that the result holds for $i-1$.
If $i$ is odd then
\begin{alignat*}{3}
&&&\;qz_i\theta(z_{i-2})-z_{i-1}\theta(z_{i-1})&\\
&=&&\;q(z_{i-1}x_i-z_{i-2})\theta(z_{i-2})-z_{i-1}\theta(z_{i-2}x_{i-1}-z_{i-3})&\\
&=&&\;z_{i-1}\theta(qx_{i-1}z_{i-2}-z_{i-2}x_{i-1}+z_{i-3})-qz_{i-2}\theta(z_{i-2})&\\
&=&&\;qz_{i-1}\theta(z_{i-3})-qz_{i-2}\theta(z_{i-2})&\text{ (by \ref{normalelementsthm})}\\
&=&&\;{-q^\frac{i-1}{2}}.&
\end{alignat*}
If $i$ is even then
\begin{align*}
&\;z_i\theta(z_{i-2})-z_{i-1}\theta(z_{i-1})&\\
=&\;(z_{i-1}x_i-z_{i-2})\theta(z_{i-2})-z_{i-1}\theta(z_{i-2}x_{i-1}-z_{i-3})&\\
=&\;z_{i-1}\theta(x_{i-1}z_{i-2}-z_{i-2}x_{i-1}+z_{i-3})-z_{i-2}\theta(z_{i-2})&\\
=&\;qz_{i-1}\theta(z_{i-3})-z_{i-2}\theta(z_{i-2})&\text{ (by \ref{normalelementsthm})}\\
=&\;{-q^\frac{i-2}{2}}.&
\end{align*}
The result follows by induction on $i$.
\end{proof}

\begin{prop}\label{Cloc2}
Let $n\geq 3$ be odd and let $q\in \K^*$. Suppose that $q$ is not a root of unity.
 Let $\alpha$ be the $\K$-automorphism of $T_{n-2}^q$ such that, for $1\leq i\leq n-2$,
$\alpha(z_i)=z_i$ if $i$ is even and
$\alpha(z_i)=q^{-1}z_i$ if $i$ is odd.
\begin{enumerate}
\item The localization $S$ of $C_n^q$ at $\mathcal{Z}_{n-2}$
    is the ambiskew polynomial algebra  $R(T_{n-2}^q,\alpha,v,1)$ where $v=(1-q)(q^\frac{n-3}{2}z_{n-2}^{-1}-z_{n-2})$,
$x=\theta^{-1}(z_{n-1})$ and $y=z_{n-2}^{-1}z_{n-1}$.

\item For any $\lambda\in \K$, the element $q^{\frac{n-3}{2}}z_{n-2}^{-1}+\lambda+qz_{n-2}$ is splitting, $\Omega-\lambda$ is a central Casimir element and $S/(\Omega-\lambda)S$ is a generalized Weyl algebra over $T_{n-2}^q$.
\end{enumerate}
\end{prop}
\begin{proof}
(i)  By \cite[Lemma 1.4]{g}, $T_{n-2}^q[x_{n-1};\tau_{n-1},\partial_{n-1}][x_n;\tau_n,\partial_n]$ is the localization of $C_n^{q}$ at $\mathcal{Z}_{n-2}$.
Observe that
$y=x_{n-1}-z_{n-2}^{-1}z_{n-3}$
 and that, by Lemma~\ref{altz},
 \begin{alignat*}{3}
 x&=&&\; \theta^{-1}(z_{n-1})=x_nz_{n-2}-\theta(z_{n-3})&\\
 &=&&\; qz_{n-2}x_n+(1-q)\theta(z_{n-3})-\theta(z_{n-3}) & \text{ (by \ref{xzC}(i))}\\
&=&&\; q(z_{n-2}x_n-\theta(z_{n-3})).&
 \end{alignat*}

 As $qz_{n-2}$ is a  unit in $T_{n-2}^q$, it follows from Lemma~\ref{sprgenchange} that $S$ is an iterated skew polynomial ring of the form $T_{n-2}^q[y;\tau_{n-1}^\prime,\partial_{n-1}^\prime][x;\tau_n^\prime,\partial_n^\prime]$ over $T_{n-2}^q$.
By Corollary~\ref{znnormal}, $yx_i=q^{(-1)^i}x_iy$ for $1\leq i\leq n-2$ and it follows that $\tau_{n-1}^\prime=\alpha$ and $\partial_{n-1}^\prime=0$.
Also by Corollary~\ref{znnormal}, $z_{n-1}x_{i+1}=q^{(-1)^{i+1}}x_{i+1}z_{n-1}$ for $1\leq i\leq n-2$. Applying $\theta^{-1}$,
we see that
 $xx_i=q^{(-1)^{i+1}}x_ix$ for $1\leq i\leq n-2$ and hence that the restrictions of $\tau_{n}^\prime$ and $\partial_{n}^\prime$ to
$T_{n-2}^q$ are $\alpha^{-1}$ and $0$ respectively. It remains to show that
$\tau_{n}^\prime(y)=y$ and that $\partial_{n}^\prime(y)=v$.

By Lemma~\ref{xzC}(i), $\tau_n(z_{n-1})=z_{n-1}$ and $\tau_n(z_{n-2})=qz_{n-2}$  and,  by \ref{znnormal}, $\gamma_{qz_{n-2}}(z_{n-1})=qz_{n-1}$. It follows,
by Lemma~\ref{sprgenchange}, that
$\tau_{n}^\prime(y)=y$.

We have seen that $\partial_{n}^\prime(t)=0$ for all $t\in T_{n-2}^q$, so $\partial_{n}^\prime(y)=\partial_{n}^\prime(x_{n-1}-z_{n-2}^{-1}z_{n-3})=\partial_{n}^\prime(x_{n-1})$.
Applying $\theta^{-1}$ to the equation in Lemma~\ref{xzC}(ii), we obtain
\begin{equation*}\label{xxn-1}
xx_{n-1}=x_{n-1}x+(q-1)(z_{n-2}-\theta^{-1}(z_{n-2})).\end{equation*}
Here   $\theta^{-1}(z_{n-2}))\notin  T_{n-2}^q[y;\tau_{n-1}^\prime,\partial_{n-1}^\prime]$ but, using Lemmas~\ref{altz} and  \ref{xzC}(i),
\begin{alignat*}{2}
\theta^{-1}(z_{n-2})&=&&\;x_nz_{n-3}-\theta(z_{n-4})\\
&=&&\;z_{n-3}x_n-q\theta(z_{n-4})\\
&=&&\;z_{n-3}(q^{-1}z_{n-2}^{-1}(x+q\theta(z_{n-3}))-q\theta(z_{n-4})
\end{alignat*}
so
\begin{equation*}
xx_{n-1}=(x_{n-1}+(q^{-1}-1)z_{n-3}z_{n-2}^{-1})x+(q-1)(z_{n-2}+q\theta(z_{n-4})-z_{n-3}z_{n-2}^{-1}\theta(z_{n-3})).\end{equation*}
As $x_{n-1}+(q^{-1}-1)z_{n-3}z_{n-2}^{-1}$ and $(q-1)(z_{n-2}+q\theta(z_{n-4})-z_{n-3}z_{n-2}^{-1}\theta(z_{n-3}))$ are both in $T_{n-2}^q[y;\tau_{n-1}^\prime,\partial_{n-1}^\prime]$, and as $z_{n-2}$ is central in $T_{n-2}^q$, it follows that
\begin{alignat*}{3}
\partial_{n}^\prime(y)=\partial_{n}^\prime(x_{n-1})& = &&\;(q-1)(z_{n-2}+q\theta(z_{n-4})-z_{n-2}^{-1}z_{n-3}\theta(z_{n-3}))&\\
& = &&\;(q-1)(z_{n-2}+z_{n-2}^{-1}(qz_{n-2}\theta(z_{n-4})-z_{n-3}\theta(z_{n-3})))\\
& = &&\;(q-1)(z_{n-2}-q^{\frac{n-3}{2}}z_{n-2}^{-1})&\text{ (by Lemma~\ref{w})}\\
& = &&\;v.&
\end{alignat*}
This completes the proof of (i).

(ii)
Let $\lambda\in \K$ and let $u=q^{\frac{n-3}{2}}z_{n-2}^{-1}+\lambda+qz_{n-2}$. Then $u$ is central in $T_{n-2}^q$ because $z_{n-2}$ is central. Also
\begin{alignat*}{2}
u - \alpha(u)&=&&\;q^\frac{n-3}{2}z_{n-2}^{-1}+\lambda+qz_{n-2} - q q^\frac{n-3}{2}z_{n-2}^{-1} -\lambda - z_{n-2}\\
&=&&\;(q-1)(z_{n-2}-q^\frac{n-3}{2}z_{n-2}^{-1})\\
&=&&\; v
\end{alignat*}
so $u$ is a splitting element.

As $z_{n-2}$ is central in $L_{n-2}^q$, the corresponding Casimir element is
\[z:=xy-u= \theta^{-1}(z_{n-1})(x_{n-1}-z_{n-3}z_{n-2}^{-1}) - q^\frac{n-3}{2}z_{n-2}^{-1} - \lambda  - q z_{n-2}.\]
Recall that $\Omega=z_{n-1}x_n - z_{n-2} - q\theta(z_{n-2})\in C_n^q$ and, from Lemma~\ref{xzC}(iii), that $\theta(\Omega)=\Omega$.
Hence $\theta^{-1}(z_{n-1})x_{n-1}=\Omega+\theta^{-1}(z_{n-2})+qz_{n-2}$ and it follows that
\[z=\Omega-\lambda+\theta^{-1}(z_{n-2})-\theta^{-1}(z_{n-1})z_{n-3}z_{n-2}^{-1}- q^\frac{n-3}{2}z_{n-2}^{-1}.\]
By Lemma~\ref{w} with $i=n-2$, $z_{n-2}\theta(z_{n-2})=z_{n-1}\theta(z_{n-3})+q^{\frac{n-3}{2}}$ so, applying $\theta^{-1}$ and postmultiplying by $z_{n-2}^{-1}$,
\[\theta^{-1}(z_{n-2})=\theta^{-1}(z_{n-1})z_{n-3}z_{n-2}^{-1}+q^{\frac{n-3}{2}}z_{n-2}^{-1}.\]
Hence $z=\Omega-\lambda$ is a Casimir element and $S/(\Omega-\lambda)S$ is a generalized Weyl algebra over $T_{n-2}^q$.
 \end{proof}

\begin{prop}\label{CSspec}
There is a bijection between $\Spec C_n^q$ and $\Spec S$ given by $P\mapsto PS$, for $P\in \Spec C_n^q$, and $Q\mapsto Q\cap C_n^q$, for $Q\in \Spec S$.
\end{prop}
\begin{proof}
For all $z\in  \mathcal{Z}_{n-2}$, $L_{n-1}^qzL_{n-1}^q=L_{n-1}^q$, by Proposition~\ref{LzL},
and so $C_{n}^qzC_{n}^q=C_{n}^q$. Hence $P\cap \mathcal{Z}_{n-2}=\emptyset$ for all $P\in \Spec C_n^q$. The result follows by \cite[Proposition 2.1.16(vii)]{McCR}.
\end{proof}

The ambiskew polynomial ring $S$ is the main example of \cite{cdfdaj}, where its prime spectrum is computed. It consists of
\begin{itemize}
\item $0$;
\item a height one prime ideal $(\Omega-\lambda)S$ for each $\lambda\in \K$;
\item for each positive integer $m$, two maximal ideals $F_{m,1}$ and $F_{m,-1}$ such that $S/F_{m,1}$ and $S/F_{m,-1}$ have Goldie rank $m$.
\end{itemize}
Here $F_{m,1}$ contains the height one prime ideal   $(\Omega-\lambda)S$, where $\lambda=(q^m+1)q^{\frac{n-2m-1}{4}}$, and $F_{m,-1}$ contains $(\Omega+\lambda)S$. Also, if $m,\ell\in \N$ are such that $m\neq\ell$ then $(q^m+1)q^{\frac{n-2m-1}{4}}\neq \pm (q^\ell+1)q^{\frac{n-2\ell-1}{4}}$.  For details, see \cite[Examples 2.8 and 3.12 and Corollary 4.7]{cdfdaj}, where $p=n-2$.

  \begin{prop}\label{Omegacompprime} For all $\lambda\in \K$, the ideal $(\Omega-\lambda)C_n^q$ is a completely prime ideal of $C_n^q$.
\end{prop}
\begin{proof} As $z_{n-1}L_{n-1}^q$ is completely prime in $L_{n-1}^q$ and  $z_{n-2}+q\theta(z_{n-2})\notin z_{n-1}L_{n-1}^q$, this follows on applying the last part of \cite[Proposition 1]{normalels} to $\Omega-\lambda=z_{n-1}x_n-z_{n-2}-q\theta(z_{n-2})-\lambda$.
\end{proof}

\begin{lemma}\label{regular} Let $R$ be a right and left noetherian ring with a right and left denominator set $\mathcal{S}$ and let $P$ be a prime ideal such that
$P\cap\mathcal{S}=\emptyset$. Then the elements of $\mathcal{S}$ are regular modulo $P$ and their images in $R/P$ form a right and left Ore set with right and left quotient ring $R_\mathcal{S}/PR_\mathcal{S}$.
\end{lemma}
\begin{proof}
Let $J=\{r\in R:rs\in P\text{ for some }s\in \mathcal{S}\}$. By a standard argument on the right Ore condition, $J$ is an ideal of $R$. As $R$ is
left noetherian, $J=Rj_1+\dots+R j_n$ for some $j_1,\dots,j_n\in R$ such that $j_is_i\in P$ for some $s_1,...,s_n\in \mathcal{S}$. By \cite[Lemma 2.1.8]{McCR} with $r_1=\dots=r_n=1$, there exists $s\in \mathcal{S}$ such that each $j_is\in P$ and hence such that $Js\subseteq P$. As $J$ is an ideal containing the prime ideal $P$ and $s\notin P$, it follows that $J=P$. By symmetry, $\{r\in R:sr\in P\text{ for some }s\in \mathcal{S}\}=P$, whence the elements of $\mathcal{S}$ are regular modulo $P$. For the rest, it is easy to check that
$\overline{\mathcal{S}}=\{s+P:s\in \mathcal{S}\}$ is a right and left Ore set in $R/P$ and that $R_\mathcal{S}/PR_\mathcal{S}$ is a right and left quotient ring of $R/P$ with respect to $\overline{\mathcal{S}}$.
\end{proof}

\begin{theorem}\label{classspecC} Suppose that $q$ is not a root of unity.
The prime spectrum of $C_n^q$ consists of $0$, the height one prime ideals $(\Omega-\lambda)C_n^q$, $\lambda\in \K$, and, for each positive integer $m$, two height two prime ideals $M_{m,1}$ and $M_{m,-1}$ for which the factors $C_n^q/M_{m,1}$ and $C_n^q/M_{m,-1}$ have Goldie rank $m$.

If $\lambda\neq \pm (q^m+1)q^{\frac{n-2m-1}{4}}$ for all $m\in \N$ then $(\Omega-\lambda)C_n^q$ is maximal. If $\lambda=(q^m+1)q^{\frac{n-2m-1}{4}}$ then
$(\Omega-\lambda)C_n^q\subset M_{m,1}$ and
$(\Omega+\lambda)C_n^q\subset M_{m,-1}$.
\end{theorem}
\begin{proof} It follows from Proposition~\ref{CSspec}, the above specification of $\Spec S$ and \cite[Proposition 2.1.14]{McCR} that the prime ideals of $C_n^q$ are as listed. Lemma~\ref{regular} ensures that \cite[Lemma 2.2.12]{McCR} is applicable to show that $S/F_{m,\pm 1}$ and $C_n^q/(F_{m,\pm 1}\cap C_n^q)$ have the same Goldie rank $m$.
\end{proof}

\begin{cor} Suppose that $q$ is not a root of unity.
If $n\geq 3$ then the cyclic connected quantized Weyl algebra $C_n^q$ is a UFD in which all but countably many  height one prime ideals are maximal.
\end{cor}

\begin{cor}\label{primitivespecC} Suppose that $q$ is not a root of unity.
The primitive spectrum of $C_n^q$ consists of the non-zero prime ideals listed in Theorem~\ref{classspecC}.
\end{cor}
\begin{proof}
By Proposition~\ref{ringprops}(ii), the ideal $0$ of $C_n^q$ is not primitive due to the existence of the central element $\Omega$. All other prime ideals of $C_n^q$ are primitive. For the non-maximal height one prime ideals this is a consequence of Proposition~\ref{ringprops}(iii) and the fact that each of the non-maximal height one prime ideals is strictly contained in a unique height two prime ideal, $F_{m,1}$ or $F_{m,-1}$ for an appropriate value of $m$.
\end{proof}

\section{Automorphisms}
In this section we determine the automorphism groups of the algebras $L_n^q$ and $C_n^q$ when $q\neq\pm1$.
We have already observed, in Proposition~\ref{someautos}, the
$\K$-automorphisms $\iota_\nu$ of $L_n^q$ such that $\iota_\nu(x_i) = \nu^{(-1)^i} x_i$ for $1\leq i\leq n$, $\nu\in \K^*$,
and, when $n$ is odd, the $\K$-automorphisms $\iota$  and $\theta$ of $C_n^q$ such that  $\iota(x_i) = -x_i$ and $\theta(x_i)=x_{i+1}$ for $1\leq i\leq n$, where indices are interpreted modulo $n$.  Clearly $\theta\iota=\iota\theta$.

\begin{theorem}\label{CAut} If $n\geq 3$ is odd and $q\neq\pm1$ then
$\aut_\K(C_n^q)$ is cyclic of order $2n$, generated by $\iota\theta$.
\end{theorem}
\begin{proof}
We interpret the indices in the $x_i$'s modulo $n$.
As $\theta\iota=\iota\theta$, $\theta$ has odd order $n$ and $\iota$ has order $2$, the subgroup $\langle \iota,\theta \rangle$ of $\aut_K(C_n^q)$ is cyclic of order $2n$, generated by $\iota\theta$.

Let $\psi\in \aut_\K(C_n^q)$. As $\psi$ must send height 1 primes to height 1 primes and $U(C_n^q)=\K$ it follows from
Theorem~\ref{classspecC}
 that  $\psi(\Omega) = \mu(\Omega - \lambda)$ for some $\mu \in \K^*$, $\lambda \in \K$. As $\Omega$ has total degree $n$ and degree $1$ in each $x_i$, there must be a permutation $\pi\in S_n$ and scalars
$\mu_i\in \K^*$ and $\lambda_i\in \K$ such that, for $1\leq i\leq n$,
$\psi(x_i) = \mu_i x_{\pi(i)} + \lambda_i$.

For $r,s\in C_n^q$, let $[r,s]_q=rs-qrs$. The possibilities for $[x_k,x_\ell]_q$, $1\leq k,\ell\leq n$, are as follows. If $\ell=k+1$ then $[x_k,x_\ell]_q=1-q$, if $\ell=k-1$ then
$[x_k,x_\ell]_q=(1-q^2)x_k x_{k-1}+q^2-q=(q^{-1}-q)x_{k-1}x_k+1-q^{-1}$, otherwise, $[x_k,x_\ell]_q \in \K x_k x_\ell=\K x_\ell x_k$.
For $1\leq i\leq n$,
\begin{eqnarray*}
1-q&=&\psi([x_i,x_{i+1}]_q)\\
&=&[\psi(x_i),\psi(x_{i+1})]_q\\
&=&\mu_i\mu_{i+1}[x_{\pi(i)},x_{\pi(i+1)}]_q+(1-q)(\lambda_i\mu_{i+1}x_{\pi(i+1)}+\lambda_{i+1}\mu_ix_{\pi(i)}+\lambda_i\lambda_{i+1}).
\end{eqnarray*}
From the possibilities listed above, we see that $\lambda_i=0$ for all $i$ so $1-q=\mu_i\mu_{i+1}[x_{\pi(i)},x_{\pi(i+1)}]_q$.
 Thus $[x_{\pi(i)},x_{\pi(i+1)}]_q\in K^*$ but, as $q\neq \pm 1$, the only possibility is that $\pi(i+1)=\pi(i)+1$. Hence each $\mu_i\mu_{i+1}=1$ and $\pi=(1~2~3~\dots~n)^{\pi(1)-1}$. Also $\mu_1^{-1}=\mu_n=\mu_{n-1}^{-1}=\mu_{n-2}=\dots\mu_2^{-1}=\mu_1$ so $\mu_1=\pm 1$ and either $\mu_i=1$ for all $i$ or $\mu_i=-1$ for all $i$.
Thus $\psi=\theta^{\pi(1)-1}$ or $\psi=\iota\theta^{\pi(1)-1}$, whence $\psi\in \langle \iota,\theta \rangle=\langle \iota\theta \rangle$ and therefore
$\aut_K(C_n^q)=\langle \iota\theta \rangle$.
\end{proof}

\begin{theorem}\label{LAut}
If $n\geq 2$ and $q\neq \pm1$ then the map $\Gamma:\K^*\rightarrow \aut_\K(L^q_n)$ given by $\nu\mapsto \iota_\nu$ is a group isomorphism.
\end{theorem}
\begin{proof}
Certainly $\Gamma$ is an injective homomorphism. We proceed as in the proof of Theorem~\ref{CAut} with $z_n$, which has total degree $n$ and degree $1$ in each $x_i$, replacing $\Omega$. The height one primes have the form $(z_n-\lambda)L_n^q$, with $\lambda=0$ if $n$ is odd, so if $\psi\in \aut_K(L_n^q)$ then $\psi(x_i) = \mu_i x_{\pi(i)} + \lambda_i$ for  some permutation $\pi \in S_n$ and some scalars $\mu_i \in \K^*$, $\lambda_i \in \K$.  Proceeding as in
\ref{CAut}, we find that $1-q=\mu_i\mu_{i+1}[x_{\pi(i)},x_{\pi(i+1)}]_q$ but only for $1\leq i\leq n-1$. If $1\leq k\leq n-1$ then $[x_n,x_k]_q\notin \K\backslash\{0\}$, as $q\neq\pm1$, so $\pi(i)\neq n$ for $1\leq i\leq n-1$. Hence $\pi(n)=n$ and it follows successively that $\pi(n-1)=n-1,\dots,\pi(1)=1$, so $\pi$ is the identity permutation. Although, unlike the cyclic case, it is not necessary that $\mu_n\mu_{1}=1$, we do have
$\mu_i\mu_{i+1}=1$ for $1\leq i\leq n-1$ and so  $\psi=\iota_\nu$, where
$\nu=\mu_1^{-1}$. Thus $\Gamma$ is an isomorphism.
\end{proof}

\section{Quantum cluster algebras}\label{qca}
In this section we present two classes of quantum cluster algebras  in terms related to the connected quantized Weyl algebras $L_n^q$ and  $C_n^q$ when $n$ is odd. We shall not give a full account of the theory of quantum cluster algebras. Helpful references include \cite{BerZel}, where the theory was first developed, \cite{GrabLaun} and, although we will not exploit its multiparameter aspect, \cite{goodyak}. The quantum cluster algebras that we consider here will have no frozen variables. In published definitions of uniparameter quantum cluster algebras, the base ring may be the ring $\Z[q^{\pm \frac{1}{2}}]$ for an indeterminate $q$, as in \cite{BerZel}, or a field which may be either $\Q(q^\frac{1}{2})$, with $q$ as before, as in \cite{GrabLaun}, or, more generally, an arbitrary field $\K$ with a distinguished element $q$ that has a square root in $\K$ and is not a root of unity, as in \cite{goodyak}. We shall take the last of these approaches..

Let $n\geq 3$ be odd. Consider the Dynkin quiver $\Delta_{n-1}$ of type $A_{n-1}$, oriented as shown:
\[1\rightarrow 2\rightarrow 3\rightarrow\dots\rightarrow n-2\rightarrow n-1.\]
It is known that all quivers
with the Dynkin diagram of type $A_{n-1}$ as underlying graph are mutation equivalent, see \cite[Lemma 3.23]{gsv}. The adjacency matrix is $B=(b_{ij})$ where, for $1\leq i\leq j\leq n-1$, $b_{ij}=1$ if $j=i+1$ and $b_{ij}=0$ if $j\neq i+1$. The inverse of $B^T$ is the skew symmetric $(n-1)\times (n-1)$ integer matrix $\Lambda=(\lambda_{ij})$ where, as in Corollary~\ref{znnormal},
$\lambda_{ij}=
0$ if $j$ is odd  or if $j$ and $i$ are both even, and
$\lambda_{ij}=1$ if $j$ is even and $i$ is odd.
Hence there is a quantum cluster algebra $\mathcal{A}_{n-1}^q$ on $A_{n-1}$ and, with each $z_jz_i=q^{\lambda_{ij}}z_iz_j$,
$(\Delta_{n-1},\{z_1,z_2,\dots,z_{n-1}\})$ is an initial seed, where $z_1,\dots,z_{n-1}$ are as in Notation~\ref{WnTn}. It will be convenient to amend the quantum cluster $\{z_1,z_2,\dots,z_{n-1}\}$ to $\{y_1,y_2,\dots,y_{n-1}\}$, where $y_i=q^{\frac{1-i}{4}}z_i$ if $i$ is odd and $y_i=q^{\frac{-i}{4}}z_i$ if $i$ is even. By the Laurent phenomenon \cite[Corollary 5.2]{BerZel}, $\mathcal{A}_{n}^q$ is a $\K$-subalgebra of $T_{n-1}^q=\K[y_1^{\pm 1},\dots,y_{n-1}^{\pm 1}]$.
By Proposition~\ref{Tsimpleorpolysimple}, $T_{n-1}^q$ is simple and $T_{n}^q=T_{n-1}^q[z_{n}^{\pm 1}]$. The linear connected quantized Weyl algebra $L_{n}^q$ is a subalgebra of
$T_{n-1}^q[z_{n}]$.
\begin{prop}\label{Anqca}
The quantum cluster algebra $\mathcal{A}_{n-1}^q$  is isomorphic to $L_{n}^q/(z_{n}-q^{\frac{n-1}{4}})L_{n}^q$.
\end{prop}
\begin{proof}
Let $\varphi$ denote the composition of the $\K$-homomorphisms
\[L_{n}^q\hookrightarrow T_{n-1}^q[z_{n}]\twoheadrightarrow T_{n-1}^q[z_{n}]/(z_{n}-q^{\frac{n-1}{4}})T_{n-1}^q[z_{n}]\simeq T_{n-1}^q.\]
Note that $\varphi(x_i)=x_i$ for $1\leq i\leq n-1$ and $\varphi(x_{n})=\varphi(z_{n-1}^{-1}(z_{n-2}+z_{n}))=z_{n-1}^{-1}(z_{n-2}+q^{\frac{n-1}{4}})$.
By \cite[Theorem 7.6]{BerZel}, $\mathcal{A}_{n-1}^q$ is generated by $y_1,y_2,\dots,y_{n-1},w_2,w_3,\dots,w_{n}$ where, for $1<i\leq n$, $w_i$ is the quantum cluster variable obtained by mutation at vertex $i-1$. Note that $y_1=z_1=x_1$. By the exchange relations (see, for example \cite[2.2, p702]{GrabLaun})
\[w_2=y_1^{-1}(1+q^\frac{1}{2}y_2)=z_1^{-1}(1+z_2)=x_2=\varphi(x_2)\]
and, for $3\leq i\leq n-1$,
\[w_i=\begin{cases}y_{i-1}^{-1}(q^{-\frac{1}{2}}y_{i-2}+y_i)\text{ if $i$ is odd}\\
y_{i-1}^{-1}(y_{i-2}+q^\frac{1}{2}y_i)\text{ if $i$ is even}.\end{cases}\]
In both cases, $w_i=z_{i-1}^{-1}(z_{i-2}+z_i))=x_i=\varphi(x_i).$
Finally,
\[w_{n}=y_{n-1}^{-1}(q^{-\frac{1}{2}}y_{n-2}+1)=z_{n-1}^{-1}(z_{n-2}+q^{\frac{n-1}{2}})=\varphi(x_{n}).\]
For $1<i<n$, $y_i$ is a $\K$-linear combination of $y_{i-1}w_i$ and $y_{i-2}$, where $y_0=1$,  so $\mathcal{A}_{n-1}^q$ is generated by $y_1,w_2,w_3,\dots,w_{n}$, that is by
$\varphi(x_1),\varphi(x_2),\dots,\varphi(x_{n-1})$ and $\varphi(x_{n})$. Thus $\mathcal{A}_{n-1}^q=\varphi(L_{n}^q)$. Clearly $z_{n}-q^{\frac{n- 1}{4}}\in \ker \varphi$
so $\mathcal{A}_{n-1}^q$ is a homomorphic image of $L_{n}^q/(z_{n}-q^{\frac{n-1}{4}})L_{n}^q$. As $q$ is not a root of unity, $L_{n}^q/(z_{n}-q^{\frac{n-1}{4}})L_{n}^q$ is simple, by Theorem~\ref{specL}, so $\mathcal{A}_{n-1}^q\simeq L_{n}^q/(z_{n}-q^{\frac{n-1}{4}})L_{n}^q$.
\end{proof}

\begin{cor}\label{simnoeth}
The quantum cluster algebra $\mathcal{A}_{n-1}^q$  is simple noetherian.
\end{cor}
\begin{proof}
The noetherian condition is immediate from Proposition~\ref{Anqca} and Proposition~\ref{ringprops}(i) while simplicity is immediate from Proposition~\ref{Anqca} and Theorem~\ref{specL}.
\end{proof}

\begin{rmk}
It can be deduced, either from Theorem~\ref{Anqca} or directly using a similar proof, that $L_n^q$ is the quantum cluster algebra of $\Delta_n$ if the vertex $n$ is frozen.
\end{rmk}

We continue to fix an odd integer $n\geq 3$ and a scalar $q\in K^*$ that is not a root of unity.
We now aim to explain the connection between the cyclic connected quantized Weyl algebra $C_n^{q^2}$ and the quantum cluster algebra $Q_q$ of the quiver denoted $P_{n+1}^{(1)}$ in the classification of periodic mutation by Fordy and Marsh \cite{fordymarsh}. We shall express this quantum cluster algebra as a quotient $U/\Delta U$ of an iterated skew polynomial extension $U$ of $\K$, with $\Delta$ central in
$U$.  Both $U$ and $U/\Delta U$ contain $C_n^{q^2}$ as a subalgebra.

The quiver $P_{n+1}^{(1)}$
 has $n+1$ vertices, which we  label $0$ to $n$, and adjacency matrix
\[B=\left(\begin{array}{cccccccc}
0&1&0&0&\dots&0&0&1\\
-1&0&1&0&\dots&0&0&0\\
0&-1&0&1&\dots&0&0&0\\
0&0&-1&0&\dots&0&0&0\\
\vdots&\vdots&\vdots&\vdots&\ddots&\vdots&\vdots&\vdots\\
0&0&0&0&\dots&-1&0&1\\
-1&0&0&0&\dots&0&-1&0
\end{array}\right)=(b_{ij}),\text{ where }\]
\[b_{ij}=\begin{cases}1&\text{ if }0\leq i\leq n-1\text{ and }j=i+1\text{ or }i=0\text{ and }j=n,\\
-1&\text{ if }1\leq i\leq n\text{ and }j=i-1\text{ or }i=n\text{ and }j=0,\\
0&\mbox{ otherwise.}
\end{cases}\]

 The labelling from $0$ will also be applied to the rows and columns of the matrices $B$ above and $\Lambda$ below and has been chosen to fit with the notation $x_1,\dots,x_n$ in a cyclic connected quantized Weyl algebra.

Figure~\ref{mutationdiag} shows $P_{6}^{(1)}$. In  general there are $n+1$ vertices $0,1,\dots,n$ with a source at $0$, a sink at $n$ and a single path from $1$ to $n$. Mutation at the source gives the same diagram but rotated so that the source is at $1$ and  the sink at $0$.
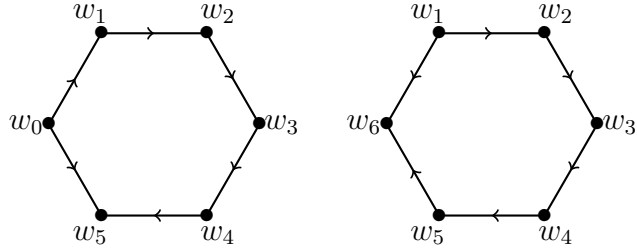
\begin{figure}[h]
\begin{center}\begin{tabular}{cc}
\begin{tikzpicture}[]
\begin{scope}[thick,decoration={
    markings,
    mark=at position 0.5 with {\arrow{>}}}
    ]
\draw[postaction={decorate}](60:1.4)--(0:1.4);
\draw[postaction={decorate}](180:1.4)--(120:1.4);
\draw[postaction={decorate}](120:1.4)--(60:1.4);
\draw[postaction={decorate}](180:1.4)--(240:1.4);
\draw[postaction={decorate}](0:1.4)--(300:1.4);
\draw[postaction={decorate}](300:1.4)--(240:1.4);
\end{scope}
\draw (0:1.7) node{$w_3$};
\draw (60:1.7) node{$w_2$};
\draw (120:1.7) node{$w_1$};
\draw (180:1.7) node{$w_0$};
\draw (240:1.7) node{$w_5$};
\draw (300:1.7) node{$w_4$};

\draw (0:1.4) node{$\bullet$};
\draw (60:1.4) node{$\bullet$};
\draw (120:1.4) node{$\bullet$};
\draw (180:1.4) node{$\bullet$};
\draw (240:1.4) node{$\bullet$};
\draw (300:1.4) node{$\bullet$};
\end{tikzpicture}
&
\begin{tikzpicture}[]
\begin{scope}[thick,decoration={
    markings,
    mark=at position 0.5 with {\arrow{>}}}
    ]
\draw[postaction={decorate}](60:1.4)--(0:1.4);
\draw[postaction={decorate}](120:1.4)--(180:1.4);
\draw[postaction={decorate}](120:1.4)--(60:1.4);
\draw[postaction={decorate}](240:1.4)--(180:1.4);
\draw[postaction={decorate}](0:1.4)--(300:1.4);
\draw[postaction={decorate}](300:1.4)--(240:1.4);
\end{scope}
\draw (0:1.7) node{$w_3$};
\draw (60:1.7) node{$w_2$};
\draw (120:1.7) node{$w_1$};
\draw (180:1.7) node{$w_6$};
\draw (240:1.7) node{$w_5$};
\draw (300:1.7) node{$w_4$};

\draw (0:1.4) node{$\bullet$};
\draw (60:1.4) node{$\bullet$};
\draw (120:1.4) node{$\bullet$};
\draw (180:1.4) node{$\bullet$};
\draw (240:1.4) node{$\bullet$};
\draw (300:1.4) node{$\bullet$};
\end{tikzpicture}
\end{tabular}
\end{center}
\caption{$P_6^{(1)}$ before and after mutation at $0$}
 \label{mutationdiag}
\end{figure}

Let $\Lambda=(\lambda_{ij})_{0\leq i,j\leq n}$ be the $(n+1)\times (n+1)$ skew-symmetric matrix
\[\left(\begin{array}{cccccccc}
0&1&0&1&\dots&1&0&1\\
-1&0&1&0&\dots&0&1&0\\
0&-1&0&1&\dots&1&0&1\\
-1&0&-1&0&\dots&0&1&0\\
\vdots&\vdots&\vdots&\vdots&\ddots&\vdots&\vdots&\vdots\\
-1&0&-1&0&\dots&0&1&0\\
0&-1&0&-1&\dots&-1&0&1\\
-1&0&-1&0&\dots&0&-1&0
\end{array}\right).\]
Thus
\[\lambda_{ij}=\begin{cases}0&\text{ if }i+j\text{ is even ,}\\
1&\text{ if }i<j\text{ and }i+j\text{ is odd,}\\
-1&\text{ if }i>j\text{ and }i+j\text{ is odd}.\\
\end{cases}\]
Let $T_q$ be the co-ordinate ring of  the quantum $(n+1)$-torus with generators $w_0^{\pm 1},w_1^{\pm 1},\dots,w_n^{\pm 1}$ and relations
$w_iw_j=q^{\lambda_{ij}}w_jw_i$ for $0\leq i<j\leq n$.
Thus  $w_iw_j=w_jw_i$ if $i+j$ is even and $w_iw_j=qw_jw_i$ if $i<j$ and $i+j$ is odd.
Then
$B^T\Lambda=2I_{n+1}$, so, with $\underline{w}:=(w_0,w_1,\dots,w_n)$, the triple $(\underline{w},B,\Lambda)$ is a quantum seed for the quiver $P_{n+1}^{(1)}$ with exchange matrix $B$, quasi-commutation matrix $\Lambda$ and initial quantum cluster $\underline{w}$.

\begin{notn}
With $\Lambda$ and $T_q$ as above, let $P_q$ denote the subalgebra of $T_q$ generated by $w_0,w_1,\dots,w_{n}$, that is the coordinate ring of $(n+1)$-dimensional quantum affine space with quasi-commutation matrix $\Lambda$. Let $D_q$ denote the quotient division algebra of $T_q$ (and $P_q$).
The  quantum cluster algebra $Q_q$ for $P_{n+1}^{(1)}$ is the subalgebra of $D_q$ generated by all possible cluster variables and, by the Laurent phenomenon, it is a subalgebra of $T_q$.
\end{notn}

Our next step is to use the exchange relations, \cite[(4.23)]{BerZel} or \cite[2.3]{GrabLaun}, to identify some further quantum cluster variables.
The mutation illustrated in Figure~\ref{mutationdiag} is in the direction of $w_0$, where the quiver has its source, and yields the new  quantum cluster variable
$w_{n+1}$, and the seed $\{w_1,w_2,\dots,w_{n+1}\}$, where, by the exchange relations,
\begin{equation}
w_{n+1}=w_0^{-1}(1+qw_1w_{n})=(1+q^{-1}w_1w_{n})w_0^{-1},\label{wnplusone}
\end{equation}
so that
\begin{equation}
w_0w_{n+1}=1+qw_1w_{n},\; w_{n+1}w_0=1+q^{-1}w_1w_{n}\text{ and }w_0w_{n+1}-w_{n+1}w_0=(q-q^{-1})w_1w_{n}.\label{w0wnplusone}
\end{equation}

Similarly, mutation in the direction of $w_n$, where the quiver has its sink, gives rise to the quantum cluster variable $w_{-1}$, and to the seed $\{w_{-1},w_0,w_1,\dots,w_{n-1}\}$, where
\[w_{-1}=w_n^{-1}(1+q^{-1}w_0w_{n-1})=(1+qw_0w_{n-1})w_n^{-1}.\]
Here the source moves to $n-1$ and  the sink to $n$.

A sequence of mutations where we successively mutate in the direction of the sources $w_0,w_1,\dots$, respectively the sinks $w_{n},w_{n-1},\dots$, will  rotate the quiver clockwise, respectively   anticlockwise. This gives rise to  a countable set $\{w_i\}_{i\in \Z}$ of cluster variables and a countable set of seeds $\{w_i,w_2,\dots,w_{i+n}\}$, $i\in \Z$, with a source at $w_i$ and a sink at $w_{i+n}$. For $i>n$,
\[w_{i}=w_{i-n-1}^{-1}(1+qw_{i-n}w_{i-1})=(1+q^{-1}w_{i-n}w_{i-1})w_{i-n-1}^{-1},\]
generalising the formula for $w_{n+1}$,
and, for $i<0$,
\[w_{i}=w_{i+n+1}^{-1}(1+q^{-1}w_{i+1}w_{i+n})=(1+qw_{i+1}w_{i+n})w_{i+n+1}^{-1},\]
generalising the formula for $w_{-1}$.
Straightforward calculation shows that, for $i\leq j\leq i+n$,
\begin{equation}\label{genwiwj}
w_iw_j=\begin{cases} w_jw_i& \text{ if }i+j\text{ is even},\\
qw_jw_i&\text{ if }i+j\text{ is odd}.\end{cases}
\end{equation}

Next we consider mutations in directions of vertices that are neither sources nor sinks.
For $i\in \Z$, let $x_i=w_i^{-1}(q^{-\frac{1}{2}}w_{i-1}+q^{\frac{1}{2}}w_{i+1})$. This is the new cluster variable obtained when a  seed $\{w_{j},w_{j+1},\dots,w_{j+n}$\} with $j<i<j+n$ is mutated in the direction of $w_i$.

\begin{lemma}\label{xrelations} For $i\in \Z$, let $w_i$ and $x_1$ be as specified above.
\begin{enumerate}
\item For all $i\in \Z$ and $k>0$,
\begin{alignat*}{2}
\text{\rm (a)}&\;x_ix_{i+1}-q^2x_{i+1}x_i=1-q^2,\\
\text{\rm (b)}&\;x_ix_{i+2k}-q^{-2}x_{i+2k}x_i=0,\\
\text{\rm (c)}&\;x_ix_{i+2k+1}-q^{2}x_{i+2k+1}x_i=0.
\end{alignat*}

\item For $i\in \Z$,
\begin{eqnarray}
w_ix_i=q^{-\frac{1}{2}}w_{i-1}+q^{\frac{1}{2}}w_{i+1},&&\quad x_iw_i=q^{\frac{1}{2}}w_{i-1}+q^{-\frac{1}{2}}w_{i+1},\label{xiwi}\\
x_iw_i-qw_ix_i&=&q^{\frac{1}{2}}(q^{-1}-q)w_{i+1}\text{ and }\label{wn}\\
x_iw_i-q^{-1}w_ix_i&=&q^{\frac{1}{2}}(1-q^{-2})w_{i-1}.\label{w0}
\end{eqnarray}
\item
For $1\leq i\leq n$ and $0\leq j\leq n$ with $j\neq i$,
\[x_iw_j=\begin{cases}
qw_jx_i&\text{ if }i<j \text{ and }i+j\text{ is even or }i>j\text{ and }i+j\text{ is odd,}\\
q^{-1}w_jx_i&\text{ if }i<j\text{ and } i+j\text{ is  odd or }i>j\text{ and } i+j\text{ is even.}\\
\end{cases}\]
\item For all $i\in \Z$, $x_{n+i}=x_i$.
\end{enumerate}
\end{lemma}
\begin{proof}
(i)(a) Using \eqref{genwiwj},
\begin{eqnarray*}
x_ix_{i+1}&=&(w_i^{-1}(q^{-\frac{1}{2}}w_{i-1}+q^{\frac{1}{2}}w_{i+1}))(w_{i+1}^{-1}(q^{-\frac{1}{2}}w_{i}+q^{\frac{1}{2}}w_{i+2}))\\
&=&w_{i+1}^{-1}w_{i}^{-1}(w_{i-1}w_{i}+qw_{i-1}w_{i+2}+w_{i}w_{i+1}+q^2w_{i+1}w_{i+2})
\end{eqnarray*}
whereas
\begin{eqnarray*}
x_{i+1}x_i&=&(w_{i+1}^{-1}(q^{-\frac{1}{2}}w_{i}+q^{\frac{1}{2}}w_{i+2}))(w_i^{-1}(q^{-\frac{1}{2}}w_{i-1}+q^{\frac{1}{2}}w_{i+1}))\\
&=&w_{i+1}^{-1}w_{i}^{-1}(q^{-2}w_{i-1}w_{i}+q^{-1}w_{i-1}w_{i+2}+w_{i}w_{i+1}+w_{i+1}w_{i+2})
\end{eqnarray*}
so $x_ix_{i+1}-q^2x_{i+1}x_i=1-q^2$.

(b)
\begin{eqnarray*}
x_ix_{i+2k}&=&q^{-1}w_i^{-1}w_{i+2k}^{-1}(q^{-\frac{1}{2}}w_{i-1}+q^{\frac{1}{2}}w_{i+1})(q^{-\frac{1}{2}}w_{i}+q^{\frac{1}{2}}w_{i+2})\\
&=&q^{-1}w_i^{-1}w_{i+2k}^{-1}(q^{-\frac{1}{2}}w_{i}+q^{\frac{1}{2}}w_{i+2})(q^{-\frac{1}{2}}w_{i-1}+q^{\frac{1}{2}}w_{i+1})\\
&=&q^{-2}w_{i+2k}^{-1}(q^{-\frac{1}{2}}w_{i}+q^{\frac{1}{2}}w_{i+2})w_i^{-1}(q^{-\frac{1}{2}}w_{i-1}+q^{\frac{1}{2}}w_{i+1})\\
&=&q^{-2}x_{i+2k}x_i.
\end{eqnarray*}

A similar calculation establishes (c) but it will be redundant given (iv), for example, as $n-3$ is even, $x_ix_{i+3}=x_{i+n}x_{i+3}=q^2x_{i+3}x_{i+n}=q^2x_{i+3}x_{i}$.

(ii) and (iii) are straightforward from the definition of $x_i$ and \eqref{genwiwj}.

(iv)
We show that $x_0=x_n$. For the general case, add $i$ to all the subscripts.
 \begin{eqnarray*}
x_0&=&w_0^{-1}(q^{-\frac{1}{2}}w_{-1}+q^{\frac{1}{2}}w_{1})\\
&=&w_0^{-1}(q^{-\frac{1}{2}}w_n^{-1}(1+qw_0w_{n-1})+q^{\frac{1}{2}}w_{1})\\
&=&q^{\frac{1}{2}}w_n^{-1}w_0^{-1}(1+qw_0w_{n-1})+q^{\frac{1}{2}}w_0^{-1}w_{1}\\
&=&q^{\frac{1}{2}}w_n^{-1}w_0^{-1}+q^{-\frac{1}{2}}w_n^{-1}w_{n-1}+q^{\frac{1}{2}}w_0^{-1}w_{1}\\
\end{eqnarray*}
and
  \begin{eqnarray*}
x_n&=&w_n^{-1}(q^{-\frac{1}{2}}w_{n-1}+q^{\frac{1}{2}}w_{n+1})\\
&=&w_n^{-1}(q^{-\frac{1}{2}}w_{n-1}+q^{\frac{1}{2}}w_0^{-1}(1+qw_{1}w_n))\\
&=&q^{-\frac{1}{2}}w_n^{-1}w_{n-1}+q^{\frac{1}{2}}w_n^{-1}w_0^{-1}+q^\frac{3}{2}w_n^{-1}w_0^{-1}w_{1}w_n\\
&=&q^{-\frac{1}{2}}w_n^{-1}w_{n-1}+q^{\frac{1}{2}}w_n^{-1}w_0^{-1}+q^\frac{1}{2}w_0^{-1}w_{1}=x_0.\\
\end{eqnarray*}
 \end{proof}

\begin{rmk}
By (i)(a)(b) and (ii), the subalgebra $C$ of the quantum cluster algebra $Q_q$ generated by $x_1,x_2,\dots,x_n$ is a homomorphic image of the cyclic quantized Weyl algebra $C_n^{q^2}$. We shall see later that $C=C_n^{q^2}$.
\end{rmk}

\begin{theorem}\label{Qgen}
The quantum cluster algebra $Q_q$ is generated by $w_0, w_1, x_1,x_2,\dots x_n$ and, indeed, by $w_1, x_1,x_2,\dots x_n$.
\end{theorem}
\begin{proof} Let $S$ be the subalgebra of $Q_q$ generated by the cluster variables $w_0, w_1, x_1,x_2,\dots x_n$. As the quiver $P_{n+1}^{(1)}$ is acyclic, it follows from \cite[Theorem 7.6]{BerZel} that $Q_q$ is generated by $w_{-1},w_0, w_1,\dots,w_n,w_{n+1},x_1,x_2,\dots$ and $x_n$. Although quantum cluster algebras are defined over $\Z[q^{\pm 1/2}]$ in \cite{{BerZel}}, the result in that case implies the result for quantum cluster algebras defined over $\K$.
Recall that $w_{-1}=q^{\frac{1}{2}}(w_0x_n-q^{\frac{1}{2}}w_{1})$ so $w_{-1}\in S$. Also,  for $i\geq 1$, it follows from Lemma~\ref{xrelations}(ii) that
$w_{i+1}=q^{\frac{-1}{2}}(w_ix_i-q^{\frac{-1}{2}}w_{i-1})$, from which it follows, inductively, that $w_2,w_3,\dots,w_n,w_{n+1}\in S$. Therefore
$Q_q=S$.
Finally, by \eqref{w0} with $i=1$,
\[x_1w_1-q^{-1}w_1x_1=q^{\frac{1}{2}}(1-q^{-2})w_0\] so $w_0$ may be omitted from the list of generators of $S$.
\end{proof}

\begin{theorem}\label{Q}
\begin{enumerate}
\item The subalgebra $S$ of $T_q$ generated by $w_0, w_1, x_1,x_2,\dots,x_{n-1}$ is an iterated skew polynomial extension $\K[w_0][w_1,\rho_0][x_1;\rho^\prime_1,\delta^\prime_1]\cdots[x_{n-1};\rho^\prime_{n-1},\delta^\prime_{n-1}]$.
\item There is a skew polynomial extension $U=S[x^\prime;\beta^\prime,\delta^\prime]$ of $S$ with a central element $\Delta$
such that there is a surjective $\K$-algebra homomorphism $\Gamma:U\rightarrow Q_q$ with $\Gamma(x)=x_n$ and $\Gamma(\Delta)=0$.
\item $\Delta U$ is a completely prime ideal of $U$.
\item $\Delta U$ is a maximal ideal of $U$ and $Q_q\simeq U/\Delta U.$
\item $Q_q$ is simple and (right and left) noetherian.
\item $Q_q$ is generated by $w_0, w_1, x_1,x_2,\dots,x_{n}$ subject to the relations
\begin{alignat*}{2}
w_0w_1&=qw_1w_0&\\
x_jw_0&=q^{(-1)^{j+1}}w_0x_j,&&\text{ if }1\leq j<n,\\
x_jw_1&=q^{(-1)^{j}}w_1x_j,&& \text{ if }1<j\leq n,\\
x_1w_1&=q^{-1}w_1x_1+q^\frac{1}{2}(1-q^2)w_0,&&\\
x_nw_0&=qw_0x_n+q^\frac{-1}{2}(1-q^2)w_1,&&\\
x_ix_{i+1}&=q^2x_{i+1}x_i+1-q^2,&& \text{ if }1\leqslant i\leqslant n-1,\\
x_nx_1&=q^2x_1x_n+1-q^2,&&\\
x_ix_j&=q^2x_jx_i, &&\text{ if } i\geqslant 1, i+1<j\leqslant n\text{ and } j-i\text{ is odd},\\
x_ix_j&=q^{-2}x_jx_i, && \text{ if }i\geqslant 1, i+1<j< n\text{ and } j-i\text{ is even},\\
w_0w_{n+1}&=qw_1w_n+1.&&
\end{alignat*}
Here $w_n$ and $w_{n+1}$ are defined recursively using the formula $w_j=q^{-\frac{1}{2}}w_{j-1}x_{j-1}+q^{-1}w_{j-2}$, $j\geq 2$, and are linear combinations of standard monomials of the form $w_1^aw_0^bx_1^{d_1}x_2^{d_2}\dots x_n^{d_n}$, where $a=1$ and $b=0$ or $a=0$ and $b=1$ and each $d_i\leq 1$.
\end{enumerate}
\end{theorem}
\begin{proof}
For $0\leq i\leq n$, let $T_q^{(i)}$ be the subalgebra of $T_q$ generated by $w_0^{\pm 1},w_1^{\pm 1},\dots, w_i^{\pm 1}$, let $R_i$ be the subalgebra generated by $w_0,w_1^{\pm 1},\dots, w_i^{\pm 1}$ and, for $i>1$, let $S_i$  be the subalgebra generated by $w_0, w_1, x_1,x_2,\dots x_{i-1}$. Thus, as $x_{i-1}=w_{i-1}^{-1}(q^{-\frac{1}{2}}w_{i-2}+q^{\frac{1}{2}}w_{i})$, we have $S_i\subset R_i\subset T_q^{(i)}$ for $i>1$. Also $T_q^{(n)}=T_q$ and, for $0\leq i\leq n-1$, $T_q^{(i+1)}$ and $R_{i+1}$ are, respectively, the skew Laurent polynomial rings $T_q^{i}[w_i^{\pm 1};\rho_{i}]$ and $R_i[w_i^{\pm 1};\rho_{i}]$, where the automorphism $\rho_i$ of $T_q^{i}$ or $R_i$, as appropriate, is such that, for $1\leq j\leq i$, $\rho_i(w_j)=w_j$ if $i+j$ is even and $\rho_i(w_j)=q^{-1}w_j$ if $i+j$ is odd. Here, and elsewhere in the proof, we abuse notation in using the same notation for $\rho_i$ and its restriction to a subalgebra.

(i) Let $1\leq i\leq n-1$. In the skew polynomial ring $R_i[w_{i+1};\rho_i]$, $x_i$ has degree one and invertible leading coefficient. It follows from Lemma~\ref{sprgenchange} that $R_i[w_{i+1};\rho_i]=
R_i[x_{i};\rho^\prime_i,\delta^\prime_i]$ for an appropriate automorphism $\rho^\prime_i$ of $R_i$ and $\rho^\prime_i$-derivation $\delta^\prime_i$.
By Lemma~\ref{xrelations}(i), for $1\leq j\leq i-1$, $\rho^\prime_i(x_j)=q^{\pm 2}x_j$ and $\delta^\prime_i(x_j)=0$ or $1-q^{-2}$. By Lemma~\ref{xrelations}(ii) and (iii), $\rho^\prime_i(w_0)=q^{(-1)^{i+1}} w_0$,  $\delta^\prime_i(w_0)=0$,  $\rho^\prime_i(w_1)=q^{(-1)^{i}}w_1$ and  $\delta^\prime_i(w_1)=0$ unless $i=1$ in which case, by \eqref{w0}, $\delta^\prime_1(w_1)=q^\frac{1}{2}(1-q^{-2})w_0$.
Thus $\rho^\prime_i(S_{i})=S_{i}$ and $\delta^\prime_i(S_{i})\subseteq S_{i}$ for $x\in \{w_0,w_1,x_1,\dots,x_{i-1}\}$ and it follows inductively that, for $1\leq i\leq n$,
\[S_{i}=\K[w_0][w_1,\rho_0][x_1;\rho^\prime_1,\delta^\prime_1]\cdots[x_{i-1};\rho^\prime_{i-1},\delta^\prime_{i-1}].\] In particular, this holds for $i=n$ where $S_{n-1}=S$.

(ii)
 Let $A$ be the quantum $n$-torus generated by $w_1^{\pm1},\dots,w_n^{\pm1}$ and let $\alpha$ be the $\K$-automorphism of $A$ such that $\alpha(w_i)=w_i$ if $i$ is even and
$\alpha(w_i)=qw_i$ if $i$ is odd. The $\K$-subalgebra $R_n$ of $T_n^q$ has the form $A[w_0;\alpha]$ while the $\K$-subalgebra generated by $w_1^{\pm1},\dots,w_n^{\pm1}$ and $w_{n+1}$ is  $A[w_{n+1};\alpha^{-1}]$. Observe that $w_1w_n$ is central in $A$ and that $\alpha(w_1w_n)=q^2w_1w_n$. Let $u=1+q^{-1}w_1w_n$ so that $v:=u-\alpha(u)=(q^{-1}-q)w_1w_n$. Form the conformal ambiskew polynomial ring $R=R(A,\alpha,v,1)$, with $w_0$ in the role of $y$. Thus $R=A[w_0;\alpha][x;\beta,\delta]$ where $\beta(w_0)=w_0$, $\delta(w_0)=(q^{-1}-q)w_1w_n$ and, for $1\leq i\leq n$,
$\beta(w_i)=\alpha^{-1}(w_i)$ and $\delta(w_i)=0$.  Note that $x_i\in A$ if $1<i\leq n-1$ and $x_1=w_1^{-1}(q^{\frac{-1}{2}}w_0+
q^{\frac{1}{2}}w_2)\in A[w_0;\alpha]$.
The Casimir element $\Delta:=xw_0-q^{-1}w_1w_n-1=xw_0-q^{-1}w_nw_1-1$ is central in $R$.

As $w_{n+1}w_j=\beta(w_j)w_{n+1}$ for $1\leq j\leq n$, and, by \eqref{w0wnplusone}, $w_{n+1}w_0-w_0w_{n+1}=(q-q^{-1})w_1w_n-1$,
there is a $\K$-algebra homomorphism $\Psi:R\rightarrow T_q$ such that $\Psi(w_i)=w_i$ for $0\leq i\leq n$, $\Psi(x)=w_{n+1}$ and $\Psi(\Delta)=0$. Note that $\Psi(x_i)=x_i$ for $1\leq i\leq n-1$.

Let $x^\prime=w_n^{-1}(q^{\frac{1}{2}}x+q^{\frac{-1}{2}}w_{n-1})$, so that $\Psi(x^\prime)=x_n$.  By Lemma~\ref{sprgenchange}, $R=A[w_0;\alpha][x^\prime;\beta^\prime,\delta^\prime]$ for an appropriate $\K$-automorphism $\beta^\prime$ of $A[w_0;\alpha]$ and an appropriate $\beta^\prime$-derivation $\delta^\prime$ of $A[w_0;\alpha]$.

The next step is to show that $\beta^\prime(S)=S$ and $\delta^\prime(S)\subseteq S$ so that $\beta^\prime$ and $\delta^\prime$ restrict, respectively, to an automorphism and $\beta^\prime$-derivation $\delta^\prime$ of $S$, which we also denote by $\beta^\prime$ and $\delta^\prime$, giving rise to a subalgebra $U$ of $R$ of the form $S[x^\prime;\beta^\prime,\delta^\prime]$.

Let $0\leq i\leq n$. Using the formulae in Lemma~\ref{sprgenchange}, we see that
$\beta^\prime(w_i)=q^{(-1)^i}w_i$ and
\[\delta^\prime(w_i)=q^\frac{1}{2}w_n^{-1}\delta(w_i)+q^{-\frac{1}{2}} w_n^{-1}w_{n-1}w_i-q^{-\frac{1}{2}}\beta^\prime(w_i)w_n^{-1}w_{n-1}.\]
As $\delta(w_i)=0$ if $i>0$ and $\delta(w_0)=(q^{-1}-q)w_1w_n$, it follows that
$\delta^\prime(w_0)=q^\frac{1}{2}(q^{-1}-q)w_1$, $\delta^\prime(w_i)=0$ for $1\leq i\leq n-1$  and
$\delta^\prime(w_n)=q^\frac{1}{2}(1-q^{-2})w_{n-1}$.

Now let $1\leq i\leq n-1$ and recall that $x_i=w_i^{-1}(q^{-\frac{1}{2}}w_{i-1}+q^\frac{1}{2})w_{i+1}$.
It follows that $\beta^\prime(x_i)=q^{2(-1)^{i-1}} x_i$. Also $\delta^\prime(x_i)=0$ for $2\leq i\leq n-2$ whereas
\[\delta^\prime(x_1)=q^{-\frac{1}{2}}\beta^\prime(w_1^{-1})\delta^\prime(w_0)=1-q^2\]
and
\[\delta^\prime(x_{n-1})=q^{\frac{1}{2}}\beta^\prime(w_{n-1}^{-1})\delta^\prime(w_n)=1-q^{-2}.\]

The above calculations establish that for each generator $g\in \{w_0,w_1,x_1,\dots,x_{n-1}\}$ of $S$, $g$ is an eigenvector of $\beta^\prime$, with non-zero eigenvalue, and $\delta^\prime(g)\in S$. Hence there is an iterated skew polynomial extension $U=S[x^\prime;\beta^\prime,\delta^\prime]\subseteq
R$ as described above. As $U$ is generated by $w_0,w_1,x_1,\dots,x_{n-1}$ and $x^\prime$, the image $\Psi(U)$ is generated by
$w_0,w_1,x_1,\dots,x_{n-1}$ and $x_n$. By Theorem~\ref{Qgen}, $\Psi(S)=Q_q$. Note that $x=q^{-\frac{1}{2}}w_nx^\prime-q^{-1}w_{n-1}\in U$ so $\Delta\in U$. As $\Delta$ is in central in $R$ and $\Delta\in U$,
$\Delta$ is central in $U$. So (ii) holds with $\Gamma$ being the restriction to $U$ of $\Psi$.

(iii) As $x=q^{-\frac{1}{2}}(w_nx^\prime-q^{\frac{-1}{2}}w_{n-1})$ and $\Delta=xw_0-q^{-1}w_1w_n-1=w_0x-qw_1w_n-1$,
 $\Delta=q^{-\frac{1}{2}}w_0w_nx^\prime+d$ where $d=-q^{-1}w_0w_{n-1}-qw_1$. Both $w_0$ and $w_n$ are normal in $S$,
with $S/w_0S$ and $S/w_nS$ isomorphic to
 skew polynomial rings over domains, so $w_0S$ and $w_nS$ are completely prime ideals. Note that $d\notin w_0S$ and $d\notin w_nS$ by easy degree arguments. If $r,s\in S$ are such that $rd=w_0w_ns$ then $r=w_0e$ for some $e\in S$, $ed\in w_nS$ and $e\in w_nS$, whence $r\in w_nw_0S$. Thus $d$ is regular modulo $w_0w_nS$. By \cite[Proposition 1]{normalels}, $\Delta U$ is a completely prime ideal of $U$.

(iv)
Note that  $R/\Delta R$  is the generalized Weyl algebra $W(A,\alpha,u)$.  We shall apply \cite[Theorem 5.4]{ambiskew} to see that $W(A,\alpha,u)$ is simple.  Criteria (i) and (ii) of that theorem hold because the skew Laurent polynomial ring $A[x^{\pm 1};\alpha]$ is simple, by \cite[Proposition 1.3]{McCP}, whence $A$ is $\alpha$-simple and $\alpha^m$ is outer for all $m\geq 1$. Criterion (iii), the regularity of $u$, is clear and (iv) holds because, for $m\geq 1$, $uA+\alpha^m(u)A$ contains the unit $(1-q^{2m})w_1w_n$. By \cite[Theorem 5.4]{ambiskew}, $W(A,\alpha,u)$ is simple.

Let $P$ be the subalgebra of $A$ generated by $w_1,\dots,w_n$. The automorphism $\alpha$ of $A$ restricts to an automorphism of $P$, which we also denote $\alpha$, and $v$ is central in $P$ so we can form the ambiskew polynomial ring $B:=R(P,\alpha,v,1)$. Thus $B\subset U\subset R$.
The multiplicatively closed set $\mathcal{S}$ generated by $\K^*\cup\{w_i:1\leq i\leq n\}$ is a right and left Ore set  of regular elements
in $P$, with ring of quotients $A$, and it follows from \cite[Lemma 1.4]{g}, that $\mathcal{S}$  is a right and left Ore set  of regular elements
in $B$, with ring of quotients $R$. As $B\subset U\subset R$, $R$ is also the ring of quotients of $U$ with respect to $\mathcal{S}$. Let $Q$ be the completely prime ideal $\Delta U$ of $U$. By degree, $w_i\notin Q$ for $1\leq i\leq n$ so  $\Delta R\cap U=\Delta U$, by \cite[Theorem 10.20]{GW} or \cite[Proposition 2.1.16]{McCR}.

 Let $J$ be a proper ideal of $U$ strictly containing $\Delta U$.  By \cite[Proposition 2.1.16]{McCR},
 $JR$ is an ideal of $R$ containing $\Delta R$. If $JR=\Delta R$ then $J\subseteq \Delta R\cap U=\Delta U$ so
$\Delta R\subset JR$.
By the simplicity of $R/\Delta R$, it follows that $JR=R$  and hence, by \cite[Proposition 2.1.16(vi,iv)]{McCR} $J$ contains a monomial $w_1^{m_1}w_2^{m_2}\dots w_n^{m_n}$ for some non-negative integers $m_1,m_2,\dots,m_n$
with at least one $m_i$ non-zero. Define the \emph{weight} of a monomial $w_1^{m_1}w_2^{m_2}\dots w_n^{m_n}$ to be $\sum_{i=1}^n im_i$ and let
$w=w_1^{m_1}w_2^{m_2}\dots w_n^{m_n}$ be a monomial of least weight in $J$. Let $i$, $1\leq i\leq n$, be maximal such that $m_i\neq 0$. Suppose that $i=n$. In the skew polynomial ring $U$, $\beta^\prime(w_n)=q^{-1}(w_n)$ and $\delta^\prime(w_n)=\mu w_{n-1}$, where $\mu=q^\frac{1}{2}(1-q^{-2})\in \K^*$. We claim that, for $k\geq 1$, $\delta^\prime(w_n^k)=\qnum{k}{q^{-2}}\mu w_{n-1}w_n^{k-1}$. This holds for $k=1$ and if it holds for $k=m-1$ then
\begin{eqnarray*}
\delta^\prime(w_n^m)&=&\beta^\prime(w_n^{m-1})\delta^\prime(w_n)+\delta^\prime(w_n^{m-1})w_n\\
&=&q^{1-m}w_n^{m-1}\mu w_{n-1}+\qnum{m-1}{q^{-2}}\mu w_{n-1}{w_n}^{m-2}w_n\\
&=&\mu(\qnum{m-1}{q^{-2}}+q^{-2(m-1)}) w_{n-1}{w_n}^{m-1}\\
&=&\mu \qnum{m}{q^{-2}} {w_{n-1}}{w_n}^{m-1}.
\end{eqnarray*}
By induction, it holds for all $m$.
With $w$ as above and $w^\prime=w_1^{m_1}w_2^{m_2}\dots w_{n-1}^{m_{n-1}}$, so that $w=w^\prime w_n^{m_n}$, $\beta^\prime(w^\prime)=q^c w^\prime$ for  some $c\in \Z$ and $\delta^\prime(w^\prime)=0$.
Hence $\delta^\prime(w)=q^cw^\prime\mu\qnum{m_n}{q^{-2}}{w_{n-1}}w_n^{m_n-1}$.
Now $\delta^\prime(w)=x^\prime w-\beta^\prime(w)x^\prime\in J$ as $\beta^\prime(w)=q^{c-m_n} w$.
But $q^c\mu\qnum{m_n}{q^{-2}}\neq 0$ and $w^\prime w_{n-1}w_n^{m_n-1}$ has weight one lower than that of $w$
so, by minimality of $w$, $m_n=0$. Repeating the argument with $U,x^\prime, \beta^\prime, \delta^\prime$ replaced by $S_i,x_i,\rho_i^\prime,\delta_i^\prime$, $i=n-1,\dots,2,1$, for which,  by \eqref{w0}, $\rho_i^\prime(w_i)=q^{-1}(w_i)$ and $\delta_i^\prime(w_i)=\mu w_{i-1}$, we see that $m_{n-1}=\dots=m_2=0$, so that $w=w_1^{m_1}$, and that $w_0w_1^{m_1-1}\in J$. Recall that $\Delta=xw_0-q^{-1}w_nw_1-1\in J$ from which it follows that
$xw_0w_1^{m_1-1}-q^{-1}w_nw_1^m-w_1^{m_1-1}\in J$ and hence that $w_1^{m_1-1}\in J$, contradicting the minimality of $w$. This contradiction shows that $\Delta U$ is a maximal ideal  of $U$. But $\Delta U\subseteq \ker \Gamma$ so $\Delta U=\ker \Gamma$  and $U/\Delta U\simeq \Gamma(U)=Q_q$.

(v) The simplicity of $Q_q$ is immediate from (iv) and the noetherian conditions follow from (ii), (iv) and Hilbert's Basis Theorem for skew polynomial rings  \cite[Theorem 2.6]{GW}.

(vi) Follows from (ii), (iv) and \cite[Proposition 1]{fddaj}.

\end{proof}

\begin{cor}\label{Qauto}
There  is a $\K$-automorphism $\theta$ of $Q_q$ such that, for $i\in\Z$,  $\theta(x_i)=x_{i+1}$,
$\theta^{-1}(x_i)=x_{i-1}$, $\theta(w_i)=w_{i+1}$ and $\theta^{-1}(w_i)=w_{i-1}$.
\end{cor}
\begin{proof}
This follows easily from Theorem~\ref{Q}(vi) and Lemma~\ref{xrelations}.
\end{proof}
\begin{rmk}
The automorphism $\theta$ in Corollary \ref{Qauto} lifts to  a $\K$-automorphism $\Theta$ of the iterated skew polynomial ring $U$ in Theorem~\ref{Q}(ii) with $\Theta(x_i)=x_{i+1}$ for $1\leq i\leq n-1$, $\Theta(x^\prime)=x_1$,
$\Theta(w_0)=w_1$, $\Theta(w_1)=q^{-\frac{1}{2}}(w_1x_1-q^{-\frac{1}{2}}w_0)$ and $\Theta(\Delta)=\Delta$. We leave the proof to the interested reader.

\end{rmk}

\begin{rmk} The cyclic connected quantized Weyl algebra $C_n^{q^2}$ embeds in $U$ in an obvious way with $x_n\mapsto x^\prime$. As $\Omega-\lambda$ cannot, by degree, be in $\Delta U$,   it follows from Theorem~\ref{Q} that this induces an embedding $C_n^{q^2}\hookrightarrow Q_q$. Also $C_n^{q^2}$ is a homomorphic image of $U$, being isomorphic to $U/I$, where $I=w_0U+w_1U$, which, by Theorem~\ref{Q}(vi), is an ideal of $U$.
\end{rmk}

\begin{rmk}

It will be shown in the PhD thesis of first author  that the set of quantum cluster variables in $Q_q$ is the union of $n$ $\theta$-orbits, namely the infinite orbit $\{w_i:i\in \Z\}$ and the $n-1$ finite $\theta$-orbits $\{\theta^{j}(z_k):0\leq j\leq n-1\}$, $1\leq k\leq n-1\}$, where the elements $z_k$ are as in  Section 5.
  \end{rmk}

\section{Poisson structures}\label{poisson} The connected quantized Weyl algebras $L_n^q$ and $C_n^q$ are quantizations, in the sense of \cite[III.5.4]{BGl},
of Poisson algebras. In other words there are Poisson brackets on the polynomial algebra $\K[x_1,\dots,x_n]$ that are semiclassical limits of the families $L_n^q$ and $C_n^q$. The Poisson algebra that is the semiclassical limit of $C_n^{q^2}$ was introduced by Fordy \cite{fordy} and  this sparked our interest in $C^q_n$. In this section we shall present results of an analysis of the Poisson prime spectrum of the semiclassical limits of $L_n^q$ and $C_n^q$. This analysis was carried out in parallel with that of the prime spectra of $L_n^q$ and $C_n^q$ and the methods in the two mirror each other.
Good references for Poisson algebras, Poisson ideals, Poisson prime ideals, Poisson cores and the Poisson centre include \cite{goodletzsemi}
and \cite{goodletzsemi}.

In the remainder of the paper, we assume, as before, that the field $\K$ is algebraically closed but also that it has characteristic $0$.
Let $n\in \N$ and let $F_n, H_n$ denote, respectively, the polynomial algebra $\K[x_1,\dots,x_n]$ and the Laurent polynomial algebra $\K[x_1^{\pm 1},\dots,x_n^{\pm 1}]$.

\begin{defn}\label{logcan}
Let $\Lambda=(\lambda_{ij})$ be an $n\times n$ skew-symmetric matrix over $\K$. On each of $F_n$ and $H_n$, there is a Poisson bracket, the \emph{log-canonical Poisson bracket}, such that, for $1\leq i,j\leq n$,
\[\{x_i,x_j\}=\lambda_{ij}x_ix_j.\]
Note that, for $m_1,\dots,m_n\in \Z$, \begin{equation}\label{monomialP}
\{x_i,x_1^{m_1}\dots x_n^{m_n}\}=(m_1\lambda_{i1}+\dots +m_n\lambda_{in})x_1^{m_1}\dots x_i^{m_i+1}\dots x_n^{m_n}.
\end{equation}
\end{defn}

The simplicity criterion for quantum tori given by \cite[Proposition 1.3]{McCP} has the  following Poisson analogue, where $\Pz(H_n)$ denotes the Poisson centre of $H_n$.
\begin{prop}\label{Psimple}
Let  $\Lambda=(\lambda_{ij})$ be an $n\times n$ skew-symmetric matrix over $\K$. Then, for the log-canonical Poisson bracket determined by $\Lambda$ on the Laurent polynomial algebra $H_n=\K[x_1^{\pm 1},\dots,x_n^{\pm 1}]$, the following are equivalent.
\begin{enumerate}
\item If $m_1,\dots,m_n\in \Z$ are such that $m_1\lambda_{i1}+\dots +m_n\lambda_{in}=0$ for $1\leq i\leq n$ then $m_i=0$ for all $i$.
\item $\Pz(H_n)=\K$.
\item $H_n$ is Poisson simple.
\end{enumerate}
\end{prop}
\begin{proof}
This follows from \cite[Lemma 1.2]{van}, the proof of which, although presented over $\C$, is valid for any base field of characteristic $0$.
\end{proof}

\begin{lemma}\label{simpleplust}
Let $S$ be a simple Poisson algebra over $\K$ with Poisson centre $\K$ and extend the Poisson bracket to the polynomial algebra $S[t]$ with $\{t,s\}=0$ for all $s\in S$. Then the non-zero Poisson prime ideals of $S[t]$ are the ideals of the form $(t-\lambda)S[t]$ for some $\lambda\in \K$.
\end{lemma}
\begin{proof}
As $\ch \K=0$ it follows from
 \cite[Lemma 6.2]{goodsemiclass} that the Poisson core $I$ of any maximal ideal of $S$ is prime. By Poisson simplicity, $I=0$ so $S$ is an integral domain.
It is clear that the ideals $(t-\lambda)S[t]$ are Poisson prime, with  $S[t]/(t-\lambda)S[t]\simeq S$.
Let $P$ be a non-zero proper Poisson prime ideal of $S[t]$ and let $d$ be the minimal degree in $t$ of non-zero elements of $P$. Then $d>0$ as $P\cap S$ is a Poisson ideal of $S$ and must be $0$. It is easy to verify that
\[J:=\{s\in S:st^d+s_{d-1}t^{d-1}+\dots+s_0\in P\text{ for some }s_{d-1},\dots,s_0 \in S\}\]
is a Poisson ideal in $S$. By Poisson simplicity, $1\in J$ so there exist $s_{d-1},\dots,s_0 \in S$ such that \[f:=t^d+s_{d-1}t^{d-1}+\dots+s_0\in P.\]
For each $s\in S$, $\deg(\{s,f\})<d$ so $s_{d-1},\dots,s_0\in \Pz(S)=\K$. Thus the prime ideal $P\cap \K[t]$ is non-zero and, as $\K$ is algebraically closed,  $(t-\lambda)S[t]\subseteq P$ for some $\lambda\in \K$. By the Poisson simplicity of $S[t]/(t-\lambda)S[t]$,
$P=(t-\lambda)S[t]$.
\end{proof}

In Remark~\ref{semiclass}, we observed that the semiclassical limit of the relation $xy-qyx=1-q$ is $\{x,y\}=xy-1$. A similar discussion shows
that  the semiclassical limits of the relations $xy-qyx=0$ and $xy-q^{-1}yx=0$ are given by  $\{x,y\}=xy$ and $\{x,y\}=-xy$ respectively. The semiclassical limit of $L_n^q$ is the polynomial algebra $F_n$ with the Poisson bracket given by
\begin{alignat*}{2}
\{x_i,x_{i+1}\}&=x_ix_{i+1}-1,\quad &&\text{if }1\leqslant i\leqslant n-1,\\
\{x_i,x_j\}&=x_ix_j, \quad &&\text{if }i\geqslant 1, i+1<j\leqslant n\text{ and } j-i\text{ is odd},\\
\{x_i,x_j\}&=-x_ix_j, \quad &&\text{if }i\geqslant 1, i+1<j\leqslant n\text{ and } j-i\text{ is even}.
\end{alignat*}
This can be made formal by applying the quantization procedure described in \cite[2.1]{goodsemiclass} and \cite[III.5.4]{BGl} to the algebra obtained from $L_n^q$ on replacing the parameter $q$ by a central invertible indeterminate $Q$ and taking $h=Q-1$. The Poisson algebra obtained on equipping $F_n$ with this bracket will be denoted $F_n^L$.

Similarly, for odd $n\geq 3$,
the family $C_n^q$, $q\in \K^*$, has semiclassical limit $F_n$ with Poisson bracket given by
\begin{alignat*}{2}
\{x_i,x_{i+1}\}&=x_ix_{i+1}-1,\quad &&\text{if } 1\leqslant i\leqslant n-1,\\
\{x_i,x_j\}&=x_ix_j,\quad &&\text{if } i\geqslant 1, i+1<j\leqslant n, \text{ and }j-i\text{ is odd},\\
\{x_i,x_j\}&=-x_ix_j,\quad &&\text{if } i\geqslant 1, i+1<j< n, \text{ and }j-i\text{ is even},\\
\{x_n,x_{1}\}&=x_nx_{1}-1.&&
\end{alignat*}
The Poisson algebra obtained on equipping $F_n$ with this bracket will be denoted $F_n^C$.
There is a Poisson automorphism $\theta$ of $F_n^C$, analogous to the automorphism $\theta$ of $C_n^q$ in \ref{someautos}(iv), given by $\theta(x_i)=x_{i+1}$, where subscripts are taken modulo $n$ in $\{1,2,\dots,n\}$.

\begin{notn}\label{TandR}
We specify $n$ distinguished elements $z_1, z_2,\dots,z_n$ of $F_n^L$ by the same recurrence formula as in Section 3: $z_{-1}=0, z_0=1$ and, for $i\geq 0$, $z_{i+1}=z_ix_{i+1}-z_{i-1}$.
Note that, if $T_L$ is the localization of $F_n$ at the multiplicatively closed set generated by $z_1, z_2,\dots,z_n$ then
$T_L=\K[z_1^{\pm 1}, z_2^{\pm 1},\dots,z_n^{\pm 1}]$ because, for $1\leq j\leq n$, $x_j=(z_j+z_{j-2})z_{j-1}^{-1}$. It follows that
$z_1, z_2,\dots,z_n$ are algebraically independent. Let $S_L=\K[z_1, z_2,\dots,z_n]$ and $U=\K[z_1^{\pm 1}, z_2^{\pm 1},\dots,z_{n-1}^{\pm 1}]$ and note that $U[z_n]$ is the localization of $F_n$ at the multiplicatively closed set generated by $z_1, z_2,\dots,z_{n-1}$.
\end{notn}

The formulae listed in the following result can be deduced from Lemmas~\ref{altz}, \ref{normalelementsthm} and \ref{normalelementsthm2} and Corollary~\ref{znnormal} by passing to the semiclassical limit or by direct calculation.

\begin{lemma}\label{PformulaeL} In the Poisson algebra $F_n^L$, the following hold.
\begin{enumerate}
\item
For $1\leq i\leq n$,
$z_i = x_1\theta(z_{i-1}) - \theta^2(z_{i-2})$.
\item
For $1\leq i,j\leq n$,
\[\{x_i,z_j\}=\begin{cases}
(-1)^{i+1}x_iz_j& \text{ if }j\text{ is odd and }j<i-1,\\
0&\text{ if }j\text{ is even and }j<i-1,\\
z_{i-2}-z_{i-1}x_i=-z_i&\text{ if }j\text{ is odd and } j=i-1,\\
z_{i-2}&\text{ if }j\text{ is even  and } j=i-1,\\
0&\text{ if }j\text{ is odd and }j\geq i,\\
(-1)^{i-1}x_iz_j&\text{ if }j\text{ is even and }j\geq i.
\end{cases}\]
\item
For $1\leq i<j\leq n$,
\[\{z_i,z_j\}=\begin{cases}
0&\text{ if }j\text{ is odd or } i,j \text{ are both even},\\
z_iz_j&\text{ if }j\text{ is even and } i \text{ is odd.}
\end{cases}\]
Thus $S_L$ is a Poisson subalgebra of $F_n^L$ with the log-canonical Poisson bracket determined by the $n\times n$ skew symmetric matrix $\Lambda_n=(\lambda_{ij})$  such that, for $j>i$,
\[\lambda_{ij}=\begin{cases}1&\text{ if }j \text{ is even and }i\text{ is odd},\\
0&\text{ otherwise.}\end{cases}\]
\end{enumerate}

\end{lemma}

\begin{lemma}\label{PspecLprelim}
If $n$ is odd then the  ideals $(z_n-\lambda)F_n$, $\lambda\in \K$, are Poisson prime ideals of $F_n^L$.
If $n$ is even then $z_nF_n$ and the ideals $z_nF_n+(z_{n-1}-\lambda)F_n$, $\lambda\in \K^*$ are Poisson prime ideals of $F_n^L$.
\end{lemma}

\begin{proof}
It follows easily from Lemma~\ref{PformulaeL}(ii) that the listed ideals are all Poisson.

As $z_n-\lambda=z_{n-1}x_n-z_{n-2}-\lambda$ and, by degree, $z_{n-2}-\lambda\notin z_{n-1}F_{n-1}$, it is readily checked, by induction, that
$z_n-\lambda$ is irreducible in $F_n$. It follows, as $F_n$ is a UFD, that $(z_n-\lambda)F_n$ is prime for all $\lambda$. Thus  $(z_n-\lambda)F_n$ is Poisson prime for all $\lambda$ if $n$ is odd and when $\lambda=0$ if $n$ is even.

Suppose that $n$ is even and that $\lambda\neq 0$. As $x_n\equiv \lambda^{-1}z_{n-2}\bmod (z_nF_n+(z_{n-1}-\lambda)F_n)$, there is a Poisson isomorphism between $F_n/(z_nF_n+(z_{n-1}-\lambda)F_n)$ and $F_{n-1}/(z_{n-1}-\lambda)F_{n-1}$ so $z_nF_n+(z_{n-1}-\lambda)F_n$ is Poisson prime.
\end{proof}

\begin{prop}\label{PspecL}
For $n\geq 2$, let $S_L$, $T_L$ and $U$ be as in Notation~\ref{TandR}.

\noindent {\rm (i)} If $n$ is even then $T_L$
is Poisson simple.

\noindent {\rm (ii)} If $n$ is odd then $T_L=U[z_n^{\pm 1}]$, $U$  is a Poisson simple
subalgebra of $T_L$ and $z_n$ is Poisson central.

\noindent {\rm (iii)} The non-zero proper Poisson prime ideals of $F_n^L$ are the ideals $(z_n-\lambda)F_n$, $\lambda\in \K$, if $n$ is odd,
and $z_n F_n$ and the ideals $z_nF_n+(z_{n-1}-\lambda)F_n$, $\lambda\in \K^*$, if $n$ is even.
\end{prop}
\begin{proof}
(i) This follows from Lemma~\ref{PformulaeL}(iii) and Proposition~\ref{Psimple}. In the application of the latter, which is analogous to the proof of Lemma~\ref{Tsimpleorpolysimple}, the rows of $\Lambda_n$ should be considered in the order
$2, 4, 6,\dots,n, n-1,n-3, n-5,\dots,1$.

(ii) If $n$ is odd, $z_n$ is Poisson central by Lemma~\ref{PformulaeL}(iii)  and $U$ is  Poisson simple by the even case.

(iii) By Lemma~\ref{PspecLprelim}, the listed ideals are Poisson prime.
Let $P$ be a non-zero Poisson prime ideal of $F_n$. By Lemma~\ref{PformulaeL}(ii), if $z_m\in P$ for some $m<n$ then  $z_j\in P$ for $0\leq j\leq m$ and, in particular $1=z_0\in P$. So $z_m\notin P$ for $m<n$ and $PU[z_n]$ is a Poisson prime ideal of $U[z_n]$, which, as observed in \ref{TandR}, is the localization of $F_n$ at the multiplicatively closed set generated by $z_1, z_2,\dots,z_{n-1}$.

Suppose that $n$ is odd. By (ii), $U$ is Poisson simple so, by
Lemma~\ref{simpleplust}, $PU[z_n]=(z_n-\lambda)U[z_n]$ for some $\lambda\in \K$. As $(z_n-\lambda)F_n$ is prime in $F_n$ it follows from standard localization theory, for example \cite[Theorem 5.32]{Sha}, that $P=(z_n-\lambda)F_n$.

Now suppose that $n$ is even. By (ii), $T_L$, which is the localization of $U[z_n]$ at the powers of $z_n$, is Poisson simple so $z_n\in PU[z_n]$.
If $PU[z_n]=z_nU[z_n]$ then $P=z_nF_n$ by \cite[Theorem 5.32]{Sha}, so we can assume that $PU[z_n]\supset z_nU[z_n]$. As $U[z_n]/z_nU[z_n]\simeq U$, it follows from the odd case that $PU[z_n]=z_nU[z_n]+(z_{n-1}-\lambda)U[z_n]$ for some $\lambda\in \K^*$. By \cite[Theorem 5.32]{Sha}, $P=z_nF_n+(z_{n-1}-\lambda)F_n$.
\end{proof}

We now turn our attention to the Poisson algebra $F_n^C$. Note that, for $1\leq i<n$, the Poisson subalgebra $\K[x_1,\dots,x_i]$  coincides with $F_i^L$. The Poisson brackets among the elements
$z_1,z_2,\dots,z_{n-1}$ and $x_1,x_2,\dots,x_{n-1}$ are as before.
The following Lemma can be deduced from Lemma~\ref{xzC} by passing to the semiclassical limit or by direct calculation.

\begin{lemma}\label{xzPC}
Let $n\geq 3$ be odd and let $\Omega=z_{n-1}x_n-z_{n-2}-\theta(z_{n-2})\in F_n$.
\begin{enumerate}
\item For $1\leq j\leq n-2$,
\[\{x_n,z_j\}=\begin{cases}
z_jx_n-\theta(z_{j-1})&\text{ if }j\text{ is odd},\\
-\theta(z_{j-1})&\text{ if }j\text{ is even}.
\end{cases}\]
\item $\{x_n,z_{n-1}\}=z_{n-2}-\theta(z_{n-2})$.
\item $\theta(\Omega)=\Omega$.
\item $\Omega$ is Poisson central in $F_n^C$.
\end{enumerate}
\end{lemma}

One might expect that, by analogy with $F_n^L$ when $n$ is odd, the non-zero Poisson prime ideals of $F_n^C$ would be the ideals $(\Omega-\lambda)F_n$. However there are two exceptional non-zero Poisson primes $M_\lambda$, $\lambda=\pm 1$, such that
$F_n^C/M_\lambda\simeq F_{n-2}^L/(z_{n-2}-\lambda)F_{n-2}^L$. To establish the existence of these, we shall need to  calculate  $\{z_{n-3},\theta(z_{n-3})\}$.

\begin{lemma}\label{thetazn-3}
Let $n\geq 3$ be odd. The following hold in $F_n^C$.
\begin{enumerate}
\item $\{x_1,\theta(z_{n-3})\}=-\theta^2(z_{n-4})=-x_1\theta(z_{n-3})+z_{n-2}$.
\item Let $2\leq i\leq n-2$. Then $\{x_i,\theta(z_{n-3})\}=(-1)^i x_i \theta(z_{n-3})$.
\item Let $0\leq i\leq n-3$. Then
\[\{z_i,\theta(z_{n-3})\}=\begin{cases} -\theta^{i+1}(z_{n-i-3}) &\text{ if } i\text{ is odd},\\
z_i\theta(z_{n-3})-\theta^{i+1}(z_{n-i-3}) & \text{ if } i\text{ is even}.
\end{cases}\]
\item $\{z_{n-3},\theta(z_{n-3})\}=z_{n-3}\theta(z_{n-3})-1$.

\end{enumerate}
\end{lemma}

\begin{proof}
(i) By Lemma~\ref{xzPC}(i), $\{x_n,z_{n-3}\}=-\theta(z_{n-4})$. The result follows by applying $\theta$ and using Lemma~\ref{altz}.

(ii) By Lemma~\ref{PformulaeL}(ii), $\{x_{i-1},z_{n-3}\}=(-1)^ix_{i-1}z_{n-3}$ and the result again follows on applying $\theta$.

(iii) The result is true when $i=0$, in which case $z_i=1$, and, by (i), when $i=1$, in which case $z_i=x_1$.
Let $i\geq 1$ and suppose that the result holds for $i$ and for $i-1$. If $i$ is even then
\begin{alignat*}{2}
&\{z_{i+1},\theta(z_{n-3})\}&&\\
=&\{z_ix_{i+1}-z_{i-1},\theta(z_{n-3})\}&&\\
=&-z_ix_{i+1}\theta(z_{n-3})+z_ix_{i+1}\theta(z_{n-3})-\theta^{i+1}(z_{n-i-3})x_{i+1}+\theta^i(z_{n-i-2})&&\;\text{(by (ii) and induction)}\\
=&-\theta^i(\theta(z_{n-i-3})x_1-z_{n-i-2})&&\\
=&-\theta^i(\theta^2(z_{n-i-4}))&& \;\text{(by \ref{PformulaeL}(i))}\\
=&-\theta^{i+2}(z_{n-i-4}),&&
\end{alignat*}
which is the result for $i+1$ in this case. If $i$ is odd then
\begin{alignat*}{2}
&\{z_{i+1},\theta(z_{n-3})\}&&\\
=&\{z_ix_{i+1}-z_{i-1},\theta(z_{n-3})\}&&\\
=&-x_{i+1}\theta^{i+1}(z_{n-i-3})+z_ix_{i+1}\theta(z_{n-3})-z_{i-1}\theta(z_{n-3})+\theta^i(z_{n-i-2})\;&& \text{(by (ii) and induction)}\\
=&\theta^i(-x_1\theta(z_{n-i-3})+z_{n-i-2})+\theta(z_{n-3})(z_ix_{i+1}-z_{i-1})&&\\
=&-\theta^{i}(\theta^2(z_{n-i-4}))+\theta(z_{n-3})z_{i+1}&& \text{(by \ref{PformulaeL}(i))}\\
=&-\theta^{i+2}(z_{n-i-4})+\theta(z_{n-3})z_{i+1},&&
\end{alignat*}
which is again the result for $i+1$ in this case. The result follows by induction.

(iv) This is the special case of (iii) when $i=n-3$.
\end{proof}

\begin{lemma}\label{exceptions}
\begin{enumerate}
\item Let $\lambda=\pm 1$, let $\tau_\lambda:F_n\rightarrow F_{n-2}$ be the $\K$-algebra homomorphism such that $\tau_\lambda(x_i)=x_i$ for $1\leq i\leq n-2$, $\tau_\lambda(x_{n-1})=\lambda z_{n-3}$ and $\tau_\lambda(x_{n})=\lambda \theta(z_{n-3})$, let $\pi_\lambda: F_{n-2}\rightarrow F_{n-2}/(z_{n-2}-\lambda)F_{n-2}$ be the canonical epimorphism and let $\rho_\lambda=\pi_\lambda\circ \tau_\lambda$. Then $\rho_\lambda:F_n^C\rightarrow F_{n-2}^L/(z_{n-2}-\lambda)F_{n-2}^L$ is a Poisson homomorphism.
\item For $\lambda=\pm 1$, let $M_\lambda=\ker \rho_\lambda$. Then $M_\lambda$ is a Poisson prime ideal of $F_n^C$ and is maximal as a Poisson ideal of $F_n^C$. As an ideal of $F_n$, $M_\lambda$ is generated by $z_{n-2}-\lambda$, $x_{n-1}-\lambda z_{n-3}$ and $x_{n}-\lambda \theta z_{n-3}$. Also $z_{n-1}\in M_\lambda$, $\theta(z_{n-2})-\lambda\in M_\lambda$ and $\Omega+2\lambda\in M_\lambda$.
\end{enumerate}
\end{lemma}

\begin{proof}
(i)   Write $\tau$, $\pi$ and $\rho$ for $\tau_\lambda$, $\pi_\lambda$ and $\rho_\lambda$ respectively. We need to show that $\rho(\{x_i,x_j\})=\{\rho(x_i),\rho(x_j)\}$ for $1\leq i<j\leq n$. This is clear when $j\leq n-2$. Let $j=n-1$.
If $i\leq n-3$ then
$\tau(\{x_i,x_{n-1}\})=(-1)^{i-1}\lambda x_iz_{n-3}$ and, by Lemma~\ref{PformulaeL}(ii),
\[\{\tau(x_i),\tau(x_{n-1})\}=\lambda\{x_i,z_{n-3}\}=(-1)^{i-1}\lambda x_iz_{n-3}.\] It follows immediately that $\rho(\{x_i,x_j\})=\{\rho(x_i),\rho(x_j)\}$.

Also \[\tau(\{x_{n-2},x_{n-1}\})=\tau(x_{n-2}x_{n-1}-1)=\lambda x_{n-2}z_{n-3}-1=\lambda(z_{n-2}+z_{n-4})-1\] whereas, by Lemma~\ref{PformulaeL}(ii)
\[\{\tau(x_{n-2}),\tau(x_{n-1})\}=\lambda\{x_{n-2},z_{n-3}\}=\lambda z_{n-4}.\]
As $\pi(\lambda(z_{n-2}+z_{n-4})-1-\lambda z_{n-4}=\lambda^2-1=0$, it follows, in this case also, that $\rho(\{x_i,x_j\})=\{\rho(x_i),\rho(x_j)\}$.

Now let $j=n$. If $2\leq i\leq n-2$ then a calculation similar to that in the case $j=n-1$, $i\leq n-3$,
but with Lemma~\ref{thetazn-3} rather than Lemma~\ref{PformulaeL}, shows that
$\tau(\{x_i,x_j\})=\{\tau(x_i),\tau(x_j)\}=(-1)^i\lambda x_j\theta(z_{n-3})$ and hence that
$\rho(\{x_i,x_j\})=\{\rho(x_i),\rho(x_j)\}$. This leaves the cases $i=1$ and $i=n-1$.
In the latter, $\tau(\{x_{n-1},x_n\})=\tau(x_{n-1}x_n-1)=\lambda^2z_{n-3}\theta(z_{n-3})-1=z_{n-3}\theta(z_{n-3})-1$, as $\lambda^2=1$,
and, by Lemma~\ref{thetazn-3}(iv), $\{\tau(x_{n-1}),\tau(x_n)\}=\lambda^2\{z_{n-3},\theta(z_{n-3})\}=z_{n-3}\theta(z_{n-3})-1.$ It follows that
$\rho(\{x_{n-1},x_n\})=\{\rho(x_{n-1}),\rho(x_n)\}$.

Finally,
$\tau(\{x_1,x_n\})=\tau(1-x_1x_n)=1-\lambda x_1\theta(z_{n-3})=1-\lambda(z_{n-2}+\theta^2(z_{n-4}))$, by \ref{PformulaeL}(i), whereas, by Lemma~\ref{xzPC}(i),
\[\{\tau(x_1),\tau(x_n)\}=\lambda\{x_1,\theta(z_{n-3})\}=\lambda\theta(\{x_n,z_{n-3}\})=-\lambda\theta^2(z_{n-4}).\]
As $\pi(1-\lambda(z_{n-2}+\theta^2(z_{n-4}))+\lambda\theta^2(z_{n-4}))=1-\lambda^2=0$, it follows again that $\rho(\{x_i,x_j\})=\{\rho(x_i),\rho(x_j)\}$.

(ii) It is clear that $\rho_\lambda$ is surjective and hence, by Lemma~\ref{PspecL} and the First Isomorphism Theorem for Poisson algebras, that $M_\lambda$ is
Poisson prime and maximal as a Poisson ideal. Clearly $z_{n-2}-\lambda\in M_\lambda$, $x_{n-1}-\lambda z_{n-3}\in M_\lambda$ and $x_{n}-\lambda \theta z_{n-3}\in M_\lambda$. Also, $M_\lambda\cap F_{n-2}$ is Poisson prime in $F_{n-2}^L$ and, by Lemma~\ref{PspecL}, must be $(z_{n-2}-\lambda)F_{n-2}$. Let $w=x_{n-1}-\lambda z_{n-3}$ and $y=x_{n}-\lambda \theta z_{n-3}$. Then $F_n=F_{n-2}[w,y]$ and so $M_\lambda=M_\lambda\cap F_{n-2}+wF_{n-2}+yF_{n-2}=(z_{n-2}-\lambda)F_{n-2}+wF_{n-2}+yF_{n-2}$. Also \[z_{n-1}=z_{n-2}x_{n-1}-z_{n-3}=(z_{n-2}-\lambda)x_{n-1}+\lambda x_{n-1}-z_{n-3}\in M_\lambda,\]
\[\theta(z_{n-2})-\lambda=z_{n-2}-\lambda-\{x_n,z_{n-1}\}\in M_\lambda,\]
by Lemma~\ref{xzPC}(ii), and
\[\Omega+2\lambda=z_{n-1}x_n-(z_{n-2}-\lambda)-(\theta(z_{n-2}-\lambda)\in M_\lambda.\]
\end{proof}

\begin{prop}\label{PspecC}
Let $n\geq 3$ be odd and let $S_C$ be the polynomial algebra $\K[z_1,\dots,z_{n-1},\Omega]$ and $T_C$ be the Laurent polynomial algebra $\K[z_1^{\pm 1},\dots,z_{n-1}^{\pm 1},\Omega^{\pm 1}]$.

\noindent {\rm (i)} $S_C$ is a Poisson subalgebra of $F_n^C$ and the Poisson brackets on $S_C$ and $T_C$ are the log-canonical Poisson brackets determined by $\Lambda_n$.

\noindent {\rm (ii)} $T_C=U[\Omega^{\pm 1}]$ where $U=\K[z_1^{\pm 1}, z_2^{\pm 1},\dots,z_{n-1}^{\pm 1}]$ is a Poisson simple
subalgebra of $T_C$ and $\Omega$ is Poisson central.

\noindent {\rm (iii)} The non-zero proper Poisson prime ideals of $F_n^C$ are the ideals $(\Omega-\lambda)F_n$, $\lambda\in \K$, and the two ideals $M_1$ and $M_{-1}$ from Lemma~\ref{exceptions}.

\noindent {\rm (iv)} For $\mu\in \K$, the Poisson algebra  $F_n^C/(\Omega-\mu)F_n^C$ is Poisson simple if and only if $\mu\neq \pm 2$.
\end{prop}

\begin{proof}
The proofs of (i) and (ii) are completely analogous to those of the odd part of Proposition~\ref{PspecL}, with $\Omega$ replacing $z_n$.

(iii) The ideals $M_{\pm 1}$ are Poisson prime by Lemma~\ref{exceptions}. The ideal $(\Omega-\lambda)F_n^C$ is Poisson, as $\Omega$ is Poisson central by Lemma~\ref{xzPC}(iv) and prime, as for $(z_n-\lambda)F_n$ in the odd part of Proposition~\ref{PspecL}. So all the ideals listed are Poisson prime.

Let $P$ be a non-zero Poisson prime ideal of $F_n^C$. By Lemma~\ref{PformulaeL}(ii), if $z_m\in P$ for some $m<n-1$ then  $z_j\in P$ for $0\leq j\leq m$ and, in particular $1=z_0\in P$. So $z_m\notin P$ for $m<n-1$. If also $z_{n-1}\notin P$ then $P=(\Omega-\lambda)F_n^C$ for some $\lambda\in \K$ as in the proof of the odd part of Proposition~\ref{PspecL}(iii), with $\Omega$ replacing $z_n$. So we may assume that $z_{n-1}\in P$.

By Lemma~\ref{xzPC}(ii), $z_{n-2}-\theta(z_{n-2})=\{x_n,z_{n-1}\}\in P$. So $z_{n-2}-\theta(z_{n-2})\in P\cap F_{n-1}$. By total degree,
$z_{n-2}-\theta(z_{n-2})\notin z_{n-1}F_{n-1}$ so, by  Proposition~\ref{PspecL}(iii), $P\cap F_{n-1}$, which is Poisson prime in $F_{n-1}^L$, must have the form $z_{n-1}F_n+(z_{n-2}-\mu)F_{n-1}$ for some $\mu\in \K^*$. As $z_{n-2}-\theta(z_{n-2})\in P$, we also have that
$\theta(z_{n-2})-\mu \in P$.

Let $\lambda=\mu^{-1}$.  Note that   $z_{n-2}x_{n-1}-z_{n-3}=z_{n-1}\in P$ so $x_{n-1}-\lambda z_{n-3}\in P$.
Also $\{x_n,\lambda z_{n-2}-1\}\in P$ so, by Proposition~\ref{xzPC}(i), $\lambda(z_{n-2}x_n-\theta(z_{n-3}))\in P$. Hence $x_n-\lambda\theta(z_{n-3})\in P$.
Therefore $\{x_{n-1},x_n-\lambda\theta(z_{n-3})\}\in P$ and, as $x_{n-1}\equiv \lambda z_{n-3}\bmod P$,
\[\{x_{n-1},x_n\}-\lambda^2\{z_{n-3},\theta(z_{n-3})\}\in P.\]
Using Lemma~\ref{thetazn-3}(iv),
\[x_{n-1}x_n-1-\lambda^2(z_{n-3}\theta(z_{n-3})-1)\in P. \eqno{(*)}\]
But we also have that $x_{n-1}(x_n-\lambda\theta(z_{n-3}))\in P$ so $x_{n-1}x_n-\lambda x_{n-1}\theta(z_{n-3})\in P$ and
$x_{n-1}x_n-\lambda^2 z_{n-3}\theta(z_{n-3})\in P$. Combining this with (*), $\lambda^2-1\in P$
so we must have $\lambda=\pm 1$ and $\lambda=\mu$. We now know that $P$ contains $z-\lambda$, $x_{n-1}-\lambda z_{n-3}$ and $x_n-\lambda\theta(z_{n-3})$ and it then follows from Lemma~\ref{exceptions} that $M_\lambda\subseteq P$. By the maximality of $M_\lambda$,
$P=M_\lambda$.

(iv) This is immediate from (iii) and the fact that, by Lemma~\ref{exceptions}(ii), $\Omega+2\lambda\in M_\lambda$.
\end{proof}

\section{Commutative cluster algebras with Poisson structure} Let $n\geq 3$ be odd.
Let $n\geq 3$ be odd.
In this section we aim to present the commutative cluster algebras $\mathcal{A}$ of the quivers $A_{n-1}$ and $P_{n+1}^{(1)}$ considered in Section~\ref{qca} as Poisson simple algebras $J/\Delta J$, where $J$ is a polynomial algebra with a Poisson bracket and $\Delta$ is a Poisson central element of $J$.

We first consider $P_{n+1}^{(1)}$.
If $w_0, w_1,\dots,w_n$ are the initial cluster variables then $\mathcal{A}$ is a Poisson subalgebra of the Laurent polynomial algebra $R:=\K[w_0^{\pm 1},w_1^{\pm1},\dots,w_n^{\pm 1}]$ with the log-canonical bracket such that, for $0\leq i,j\leq n$, $\{w_i,w_j\}=\lambda_{ij}w_iw_j$ where, if $j\geq i$,
\begin{equation}\lambda_{ij}=-\lambda_{ji}=\begin{cases}1&\text{ if }j-i\text{ is odd},\\
0&\text{ if }j-i\text{ is even.}
\end{cases}\label{Pww}\end{equation}
With the matrices $B$ and $\Lambda$ as in Section~\ref{qca}, \cite[Theorem 1.4]{gsv} ensures that the above Poisson bracket $\{-,-\}$  is compatible with the cluster algebra $\mathcal{A}$, in other words, for each seed $\{y_1,y_2,\dots,y_n\}$ and for $0\leq i,j\leq n$, $\{y_i,y_j\}=\lambda^\prime_{ij}y_iy_j$ for some antisymmetric $({n+1})\times({n+1})$ matrix $\Lambda^\prime=(\lambda^\prime_{ij})$.

\begin{lemma}\label{Dsimple}
The Poisson algebra $R$ is Poisson simple.
\end{lemma}
\begin{proof}
Let $m_0,m_1,\dots m_n\in \Z$ be such that $m_0\lambda_{i0}+\dots +m_n\lambda_{in}=0$ for $0\leq i\leq n$. From the cases $i=0$ and $i=n-1$,
\[m_1+m_3+\dots+m_{n-2}+m_n=0=-m_1-m_3\dots-m_{n-2}+m_n,\]
whence $m_n=0$. Similarly the cases $i=1$ and $i=n$ give
\[-m_0+m_2+\dots+m_{n-1}=0=-m_0-m_2+\dots-m_{n-1},\] and
$m_0=0$. Repeating the argument, deleting the first and last columns at each stage, gives $0=m_{n-1}=m_1=m_{n-2}=m_2=\dots=m_{(n+1)/2}$.
By Proposition~\ref{Psimple}, $R$ is Poisson simple.
\end{proof}

As in the quantum case, there are new cluster variables $w_i$ and $x_i$, $i\in \Z$,
such that, for $i>n$,
\[w_{i}=w_{i-n-1}^{-1}(1+w_{i-n}w_{i-1})\]
for $i<0$,
\[w_{i}=w_{i+n+1}^{-1}(1+w_{i+n}w_{i+1}),\]
and, for $i\in \Z$,
\[x_i=w_i^{-1}(w_{i-1}+w_{i+1}).\]

The following can be deduced from the quantum counterpart, Lemma~\ref{xrelations}, by taking the semiclassical limit or directly from \eqref{Pww}.
Parts (i) and (iv) were observed by Fordy in \cite{fordy}.
\begin{lemma}\label{xPrelations} With $x_1, x_2,\dots, x_n$ as specified above,
\begin{enumerate}
\item For all $i\in \Z$ and $k>0$,\\
{\rm (a)} $\{x_i,x_{i+1}\}=2(x_ix_{i+1}-1)$,\\
{\rm (b)} $\{x_i,x_{i+2k}\}=-2x_{i+2k}x_i$,\\
{\rm (c)} $\{x_i,x_{i+2k+1}\}=2x_{i+2k+1}x_i$.
\item For $i\in \Z$,
\begin{equation}\label{Pxiwi}
w_ix_i=w_{i-1}+w_{i+1}
\end{equation}
and
\begin{equation}
\{x_i,w_i\}=x_iw_i-2w_{i+1}=2w_{i+1}-x_iw_i\label{Pwn}
\end{equation}
\item
For $1\leq i\leq n$ and $0\leq j\leq n$ with $j\neq i$,
\[\{x_i,w_j\}=\begin{cases}
x_iw_j&\text{ if }i<j\text{ and }i+j\text{ is even or }i>j\text{ and }i+j\text{ is odd,}\\
-x_iw_j&\text{ if }i<j\text{ and }i+j\text{ is odd or }i>j\text{ and }i+j\text{ is even.}\\
\end{cases}\]
\item For all $i\in \Z$, $x_{n+i}=x_i$.
\end{enumerate}
\end{lemma}

\begin{lemma}\label{ww}
For $i\leq j\leq i+n$,
\begin{equation}
\{w_i,w_j\}=\begin{cases} 0 &\text{ if }i+j\text{ is even},\\
w_iw_j&\text{ if }i+j\text{ is odd}\end{cases}
\end{equation}
and
\begin{equation}\label{Pgenw0wnplusone}
\{w_0,w_{n+1}\}=
2w_1w_{n}.
\end{equation}
\end{lemma}
\begin{proof}
These are straightforward calculations and are omitted.
\end{proof}

By \cite[Corollary 1.21]{BerFomZel}, the  cluster algebra $\mathcal{A}$ is the subalgebra of  $R$ generated by the cluster variables $w_{-1}, w_0, w_1,\dots w_n, w_{n+1}, x_1,x_2,\dots x_n$.
By Lemma~\ref{xPrelations}(iv), $x_0=x_n$ so, by \eqref{Pxiwi}, $w_{-1}=w_0x_n-w_1$ and, for $j\geq 2$, $w_j=w_{j-1}x_{j-1}-w_{j-2}$. Hence  the list of generators can be reduced  to  $w_0, w_1, x_1,x_2,\dots x_n$. By Lemmas~\ref{xPrelations}(iv) and \ref{ww}, $\mathcal{A}$ is a Poisson subalgebra of $R$.

\begin{notn}
\label{Kplus}
Let $D=\K[W_0,W_1,\dots,W_{n+1}]$ be a polynomial algebra in $n+2$ variables and
let $E=\K[W_0^{\pm 1},W_1^{\pm 1},\dots,W_n^{\pm 1},W_{n+1}]$.  Let $J$ be the subalgebra of $E$ generated by $W_0, W_1, X_1, X_2,\dots, X_n$, where, for $1\leq i \leq n$, $X_i=W_i^{-1}(W_{i-1}+W_{i+1})$.
Observe that $W_{i+1}=W_iX_i-W_{i-1}$ from which it follows that $D\subseteq J\subseteq E$ and hence that $J$ is a polynomial algebra in $n+2$ indeterminates $W_0, W_1, X_1, X_2,\dots, X_n$.
\end{notn}

\begin{lemma}\label{BPoisson}
There is a Poisson bracket on $D$, and hence on $E$, such that, for $0\leq i<j\leq n+1$,
\begin{equation*}\{W_i,W_j\}=\begin{cases} W_iW_j &\text{ if } j-i \text{ is odd},\\
0&\text{ if } j-i \text{ is even and }j-i<n+1,\\
2W_{1}W_{n} &\text{ if } i=0\text{ and }j=n+1.\end{cases}\end{equation*}
Also $J$ is a Poisson subalgebra of $D$.
\end{lemma}

\begin{proof}

Note that the given rules for $\{W_i,W_m\}$ determine log-canonical Poisson brackets on each of $\K[W_1,\dots,W_{n+1}]$ and $\K[W_0,W_1,\dots,W_{n}]$.
To establish the Jacobi identity on $D$, it suffices to check it on the triples $(W_0,W_i,W_{n+1})$, $0<i<n+1$.
If $i$ is odd then $\{\{W_0,W_i\},W_{n+1}\}=W_0W_iW_{n+1}+2W_1W_iW_n$, $\{\{W_i,W_{n+1}\},W_0\}=-W_0W_iW_{n+1}-2W_1W_iW_n$ and
$\{\{W_{n+1},W_0\},W_{i}\}=-2\{W_1W_n,W_i\}=0$ so the Jacobi identity holds.
If $i$ is even then $\{\{W_0,W_i\},W_{n+1}\}=0=\{\{W_i,W_{n+1}\},W_0\}$ and
$\{\{W_{n+1},W_0\},W_{i}\}=-2\{W_1W_n,W_i\}=0$ so the Jacobi identity holds again.  Also, for $1\leq i\leq n$, $\{W_i,-\}$ is a derivation on $D$ and $\{W_{n+1},-\}$ is the restriction to $D$ of the derivation $\sum_{i=0}^{n+1} f_i\partial_i$,
where each $\partial_i=\frac{\partial}{\partial W_i}$, $f_0=-2W_1W_{n}$, $f_i=0$ if $i>0$ is even and $f_i=-W_1W_{n+1}$ if $i$ is odd.
As $\{r+s,-\}=\{r,-\}+\{s,-\}$ and $\{rs,-\}=r\{s,-\}+s\{r,-\}$ for, $\{p,-\}$ is a derivation for all $p\in D$. So we have a Poisson bracket on $D$ and hence, by \cite[Lemma 1.3]{kaledin}, on $E$.

The following can be deduced from the Poisson bracket on $D$.
\begin{alignat}{3}
\{X_i,X_{i+1}\}&&=&\;2(X_iX_{i+1}-1)&& \text{ if }i<n,\label{a}\\
\{X_n,X_{1}\}&&=&\;2(X_nX_{1}-1),&&\label{b}\\
\{X_i,X_{i+2k}\}&&=&\;{-2X_{i+2k}X_i}&&\text{ if }1\leq i\leq i+2k\leq n,\label{c}\\
\{X_i,X_{i+2k+1}\}&&=&\;2X_{i+2k+1}X_i&&\text{ if }1\leq i\leq i+2k+1\leq n,\label{d}\\
\{X_i,W_i\}&&=&\;{W_{i-1}-W_{i+1}}&&\label{e1}\\
&&=&\;{X_iW_i-2W_{i+1}}&&\label{e2}\\
&&=&\;{2W_{i-1}-X_iW_i},&&\label{e3} \\
\{X_n,W_0\}&&=&\;{X_nW_0-2W_{1}},&&\label{f}\\
\{X_i,W_j\}&&=&\;X_iW_j&& \text{ if } i+j\text{ is odd and }j=0\text{ or }j=1,\\
\{X_i,W_j\}&&=&\;{-X_iW_j}&& \text{ if } i+j\text{ is even, }j\neq i\text{ and }j=0\text{ or }j=1.
\label{g}
\end{alignat}

It follows that $J$ is a Poisson subalgebra of $D$.

\end{proof}

\begin{prop}\label{central}
Let $\Delta=W_0W_{n+1}-W_1W_n-1$. Then $\Delta$ is Poisson central in $D$, $E$ and $J$.  Also $\Delta D, \Delta E$ and  $\Delta J$ are Poisson prime ideals of $D, E$ and $J$ respectively and both $E/\Delta E$ and $J/\Delta J$ are Poisson simple.
\end{prop}
\begin{proof}
For $1\leq i\leq n$ either $i$ is even and $\{W_i,W_{n+1}\}=0=\{W_i,W_0\}$ and  $\{W_i,W_{n}\}W_{1}=-\{W_i,W_{1}\}W_{n}$ or $i$ is odd and
$\{W_i,W_{n+1}\}W_{0}=-\{W_i,W_{0}\}W_{n+1}$ and  $\{W_i,W_{n}\}=0=\{W_i,W_{1}\}$. In both cases, $\{W_i,\Delta\}=0$. Also
\begin{eqnarray*}
\{W_{n+1},\Delta\}&=&W_{n+1}(-2W_{n}W_{1})+2W_{n+1}W_{n}W_{1}=0\text{ and }\\
\{W_{0},\Delta\}&=&W_{0}(2W_{n}W_{1})-2W_{0}W_{n}W_{1}=0.
\end{eqnarray*}
Hence $\Delta\in \Pz(D)$. It follows that $\Delta\in \Pz(E)$ and  $\Delta\in \Pz(J)$.

As $\Delta$ is irreducible and Poisson central in $D$, $\Delta D$ is a Poisson prime ideal of $D$. Observe that $E$ is the localization of $D$, and also of the intermediate ring $J$, at the multiplicatively closed set $\mathcal{W}$ generated by $W_0, W_1,\dots,W_n$ and that $\mathcal{W}\cap \Delta D=\emptyset$. It follows that $\Delta E$ is a Poisson prime ideal of $E$ so that $\Delta$ is irreducible in $E$. As $W_0,W_1,\dots,W_n$ are, by an easy induction using the formula $W_{i+1}=W_iX_i-W_{i-1}$, prime elements of $J$, it follows that $\Delta$ is irreducible in the UFD $J$. Hence
$\Delta J$ is a Poisson prime ideal of $J$.

There is a Poisson algebra homomorphism $\phi$ from $R$ to $E/\Delta E$ given by $w_i\mapsto W_i+\Delta E$ for $0\leq i\leq n$. As
$W_{n-1}+\Delta E=\phi(w_0^{-1}w_1w_n)$, $\phi$ is surjective and as $R$ is Poisson simple, by Lemma~\ref{Dsimple}, $\phi$ is an isomorphism. Thus
$E/\Delta E$ is Poisson simple.

Suppose that $J/\Delta J$ is not Poisson simple. By \cite[3.3(ii)]{dajsqoh2}, $J$ has a Poisson primitive, and hence Poisson prime, ideal $Q$ such that $\Delta J\subset Q\subset J$. Recall that $E$ is the localization of $J$ at $\mathcal{W}$ so $QE$ is a Poisson prime ideal of $E$ strictly containing
$\Delta E$. By the Poisson simplicity of $E/\Delta E$, $QE=E$ and therefore $Q\cap \mathcal{W}\neq\emptyset$. As $Q$ is prime, $W_i\in Q$ for some $i$ with $0\leq i\leq n$. By \eqref{f}, $\{X_n,W_0\}=X_nW_0-2W_{1}$ and if $i>0$ then, by \eqref{e2} and \eqref{e3}, $\{X_i,W_i\}=X_iW_i-2W_{i+1}=2W_{i-1}-X_iW_i$. It follows that $W_j\in Q$ for $0\leq j\leq n+1$. But $W_0W_{n+1}-W_1W_n-1=\Delta\in Q$ so $1\in Q$, contradicting the fact that $Q$ is proper. Thus $J/\Delta J$ is Poisson simple.
\end{proof}

\begin{prop}
The cluster algebra $\mathcal{A}$ is isomorphic to the simple Poisson algebra $J/\Delta J$.
\end{prop}
\begin{proof} Recall the Poisson isomorphism $\phi$ from $R$ to $E/\Delta E$ such that $w_i\mapsto W_i+\Delta E$ and note that $x_i\mapsto X_i+\Delta E$. We have observed that $\mathcal{A}$ is the subalgebra of  $R$ generated by $w_0, w_1,x_1,x_2,\dots x_n$ and is a Poisson subalgebra of $R$.
 Hence $\mathcal{A}\simeq \phi(\mathcal{A})$ which is generated by
$W_0+\Delta E$, $W_1+\Delta E$ and the $n$ elements $X_i+\Delta E$, $1\leq i\leq n$. As $J\cap \Delta E=\Delta J$ by standard localization theory, $J/\Delta J$ embeds in $E/\Delta E$ by $b+\Delta J \mapsto  b+\Delta E$ and $J/\Delta J$ is generated by $W_0+\Delta J$, $W_1+\Delta J$ and the $n$ elements $X_i+\Delta J$, $1\leq i\leq n$ so $J/\Delta J\simeq \phi(A)$.  Thus $\mathcal{A}\simeq J/\Delta J$.
\end{proof}

\begin{rmks}
In accordance with the quantum case, $\mathcal{A}$ has a Poisson automorphism $\theta$ such that $\theta(w_i)=w_{i+1}$, giving rise to an infinite orbit $\{w_i\}_{i\in \Z}$, and $\theta(x_i)=x_{i+1}$, giving rise to a finite orbit $\{w_i\}_{1\leq i\leq n}$. The full set of cluster variables in $\mathcal{A}$ is the union of $n$ $\theta$-orbits, namely the infinite orbit $\{w_i:i\in \Z\}$ and the $n-1$ finite $\theta$-orbits $\{\theta^{j}(z_k):0\leq j\leq n-1\}$, $1\leq k\leq n-1\}$, where the elements $z_k$ are as in  Section 7.

The Poisson analogue of Theorem~\ref{Qgen}, where $w_0$ could be omitted from the generators,  is more subtle. Here
$w_0=(\{x_1,w_1\}+x_1w_1)/2$ so the omission of $w_0$ from the generators would require an appropriate definition of generators of a Poisson algebra.
\end{rmks}

In the case of the Dynkin quivers $A_{n-1}$ we have the following result. The proof is omitted, as is that of the subsequent analogue of Corollary~\ref{simnoeth}. It is parallel to that of Proposition~\ref{Anqca}. The change of initial cluster variables from $z_i$ to $y_i$ is redundant and Poisson simplicity of $F^L/(z_n-1)F^L$, from  Proposition~\ref{PspecL}(iii), replaces simplicity of $L_{n}^q/(z_{n}-q^{(2-n)/2})L_{n}^q$.

\begin{prop} Let $n\geq 3$ be odd and let $F^L$ be the polynomial algebra $\K[x_1,\dots,x_{n}]$ equipped with the Poisson  bracket corresponding to $L_n^q$.
The cluster algebra $\mathcal{A}_{n-1}$  of the Dynkin quiver of type $A_{n-1}$ is isomorphic, as a Poisson algebra, to $F^L/(z_{n}-1)F^L$.
\end{prop}

\begin{cor}\label{Psimnoeth} Let $n\geq 3$ be odd.
The  cluster algebra $\mathcal{A}_{n-1}$  is a simple Poisson algebra.
\end{cor}

\end{document}